\documentclass[graybox,final,A4paper]{article}

\usepackage[a4paper, total={145mm,190mm}]{geometry}
\usepackage[color,notref,notcite]{showkeys}

\usepackage{nicefrac}
\usepackage[export]{adjustbox}

\usepackage{amsmath,amssymb}
\usepackage{amsthm}
\usepackage{graphicx}

\usepackage[hidelinks]{hyperref}
\usepackage[dvipsnames]{xcolor}
\usepackage{pgfplots,pgfplotstable}
\usepackage{subcaption}
\newcommand{\coloneq}{:=}

\usetikzlibrary{calc,external}
\usetikzlibrary{spy}
\usetikzlibrary{cd}
\tikzcdset{every diagram/.append style = {font=\normalsize}}
\tikzcdset{every label/.append style = {font = \small}}
\tikzset{ shorten <>/.style={ shorten >=#1, shorten <=#1 } }

\newcommand{\logLogSlopeTriangle}[5]
{
    \pgfplotsextra
    {
        \pgfkeysgetvalue{/pgfplots/xmin}{\xmin}
        \pgfkeysgetvalue{/pgfplots/xmax}{\xmax}
        \pgfkeysgetvalue{/pgfplots/ymin}{\ymin}
        \pgfkeysgetvalue{/pgfplots/ymax}{\ymax}

        \pgfmathsetmacro{\xArel}{#1}
        \pgfmathsetmacro{\yArel}{#3}
        \pgfmathsetmacro{\xBrel}{#1-#2}
        \pgfmathsetmacro{\yBrel}{\yArel}
        \pgfmathsetmacro{\xCrel}{\xArel}

        \pgfmathsetmacro{\lnxB}{\xmin*(1-(#1-#2))+\xmax*(#1-#2)} 
        \pgfmathsetmacro{\lnxA}{\xmin*(1-#1)+\xmax*#1} 
        \pgfmathsetmacro{\lnyA}{\ymin*(1-#3)+\ymax*#3} 
        \pgfmathsetmacro{\lnyC}{\lnyA+#4*(\lnxA-\lnxB)}
        \pgfmathsetmacro{\yCrel}{\lnyC-\ymin)/(\ymax-\ymin)}

        \coordinate (A) at (rel axis cs:\xArel,\yArel);
        \coordinate (B) at (rel axis cs:\xBrel,\yBrel);
        \coordinate (C) at (rel axis cs:\xCrel,\yCrel);

        \draw[#5]   (A)-- node[pos=0.5,anchor=north] {\scriptsize{1}}
                    (B)-- 
                    (C)-- node[pos=0.,anchor=west] {\scriptsize{#4}} 
                    cycle;
    }
}

\usepackage{filecontents}

\graphicspath{{figures/pdf/}}

\renewcommand{\vec}[1]{\boldsymbol{#1}}    
\newcommand{\tens}[1]{\boldsymbol{#1}}

\newcommand{\GRAD}{\vec{\nabla}}
\newcommand{\DIV}{\vec{\nabla}{\cdot}}

\newcommand{\Id}[1][d]{I_{\!#1}}
\newcommand{\norm}[2][]{\|#2\|_{#1}}
\newcommand{\seminorm}[2][]{|#2|_{#1}}
\newcommand{\vvvert}{\vert\kern-0.25ex\vert\kern-0.25ex\vert}

\newcommand{\normml}[2][]{\mathrm{N}_{#1}(#2)}

\newcommand{\meas}[2][d]{|#2|_{#1}}

\newcommand{\term}{\mathfrak{T}}

\newcommand{\Real}{\mathbb{R}}
\newcommand{\Natural}{\mathbb{N}}
\newcommand{\Poly}[2][]{\mathbb{P}_{#1}^{#2}}

\newcommand{\Mh}[1][h]{\mathcal{M}_{#1}}
\newcommand{\Th}[1][h]{\mathcal{T}_{#1}}
\newcommand{\Fh}[1][h]{\mathcal{F}_{#1}}

\newcommand{\fTh}[1][h]{\mathfrak{T}_{#1}}
\newcommand{\fFh}[1][h]{\mathfrak{F}_{#1}}

\newcommand{\normal}{\vec{n}}

\newcommand{\map}[1][T]{\phi_{#1}}
\newcommand{\Jmap}[1][T]{J\!\phi_{#1}}
\newcommand{\trans}[1]{\widehat{#1}}
\newcommand{\tT}{\trans{T}}
\newcommand{\tF}{\trans{F}}

\newcommand{\lproj}[2][h]{\pi_{#1}^{0,#2}}

\newcommand{\dproj}[2][T]{\pi_{\diff[],#1}^{1,#2}}
\newcommand{\tdproj}[2][\trans{T}]{\pi_{\diff[],\map[],#1}^{1,#2}}

\newcommand{\UT}[1][k]{\underline{U}_T^{#1}}
\newcommand{\UtT}[1][k]{\underline{U}_{\trans{T}}^{#1}}

\newcommand{\UhD}[1][k]{\underline{U}_{h,0}^{#1}}
\newcommand{\IT}[1][k]{\underline{I}_T^{#1}}
\newcommand{\ItT}[1][k]{\underline{I}_{\trans{T}}^{#1}}
\newcommand{\Ih}[1][k]{\underline{I}_h^{#1}}

\newcommand{\uu}[1][T]{\underline{u}_{#1}}

\newcommand{\uv}[1][T]{\underline{v}_{#1}}

\newcommand{\uw}[1][T]{\underline{w}_{#1}}

\newcommand{\dpT}[1][k+1]{\mathrm{p}_{\diff[],T}^{#1}}
\newcommand{\dptT}[1][k+1]{\mathrm{p}_{\diff[],\map[],\trans{T}}^{#1}}

\newcommand{\est}[2][T]{\varepsilon_{{\rm #2}\ifthenelse{\equal{#1}{}}{}{,#1}}}

\newcommand{\diff}[1][T]{\tens{K}\ifthenelse{\equal{#1}{}}{}{_{#1}}}
\newcommand{\sdiff}[1][T]{K\ifthenelse{\equal{#1}{}}{}{_{#1}}}

\newcommand{\tdiff}[1][\trans{T}]{\tens{K}\ifthenelse{\equal{#1}{}}{_{\map[]}}{_{\map[],#1}}}
\newcommand{\ltdiff}[1][\trans{T}]{\underline{K}_{\map[],#1}}
\newcommand{\utdiff}[1][\trans{T}]{\overline{K}_{\map[],#1}}
\newcommand{\stdiff}[1][\trans{T}]{K\ifthenelse{\equal{#1}{}}{_{\map[]}}{_{\map[],#1}}}
\newcommand{\ar}[1]{\alpha_{#1}}
\newcommand{\fl}{\mathrm{fl}}

\newcommand{\email}[1]{\href{mailto:#1}{\tt #1}}

\newtheorem{theorem}{Theorem}
\newtheorem{lemma}[theorem]{Lemma}
\newtheorem{proposition}[theorem]{Proposition}

\newtheorem{remark}[theorem]{Remark}
\newtheorem{assumption}[theorem]{Assumption}
\newtheorem{definition}[theorem]{Definition}

\begin{document}

\title{Interplay between diffusion anisotropy and mesh skewness in Hybrid High-Order schemes}

\author{J\'er\^ome Droniou \thanks{School of Mathematics,
Monash University, Victoria 3800, Australia, \email{jerome.droniou@monash.edu}.}}

\maketitle

\abstract{We explore the effects of mesh skewness on the accuracy of standard Hybrid High-Order (HHO) schemes for anisotropic diffusion equations. After defining a notion of regular skewed mesh sequences, which allows, e.g., for elements that become more and more elongated during mesh refinement, we establish an error estimate in which we precisely track the dependency of the local multiplicative constants in terms of the diffusion tensor and mesh skewness. 
This dependency makes explicit an interplay between the local diffusion properties and the distortion of the elements. We then provide several numerical results to assess the practical convergence properties of HHO for highly anisotropic diffusion or highly distorted meshes. These tests indicate a more robust behaviour than the theoretical estimate indicates.
}

\section{Introduction}\label{sec:intro}

The last few years have seen a increased interest in novel discretisation methods, for diffusion equations, that support polytopal meshes (made of general polygons/polyhedra) and allow for arbitrary approximation orders: Hybridisable Discontinuous Galerkin methods \cite{Cockburn.Gopalakrishnan.ea:09}, Virtual Element Methods (VEM) \cite{Beirao-da-Veiga.Brezzi.ea:13}, Weak Galerkin Methods \cite{Wang.Ye:13}, etc. The Hybrid High-Order (HHO) method \cite{Di-Pietro.Ern.ea:14,hho-book} is one of these arbitrary-order polytopal methods, and shares with the aforementioned ones the hybrid structure of unknowns (contrary to Discontinuous Galerkin methods \cite{Di-Pietro.Ern:12}), that is, unknowns located in the elements and on their faces. We refer to the introduction of \cite{hho-book} for a thorough review of the literature on polytopal methods. The HHO method can be seen as a high-order extension of the Hybrid Mimetic Mixed method \cite{Droniou.Eymard.ea:10} and, contrary to some other polytopal methods, it has a flux formulation that makes it a Finite Volume method. Additionally, the design of HHO schemes is dimension-independent and has an enhanced compliance with the physics due to the construction of local problem-dependent reconstruction operators.

HHO schemes have been applied to and analysed for a variety of models (see \cite{hho-book} and references therein), with error estimates that have an explicit dependency on the physical data. These estimates are however obtained for ``regular'' polytopal meshes, that is, meshes whose elements are ``isotropic'' (not elongated in any particular direction, and whose faces have a diameter comparable to their elements' diameters). In this work we analyse and numerically test the HHO scheme for highly anisotropic diffusion equations and families of distorted meshes, that no longer satisfy the usual regularity conditions. We consider the archetypal linear diffusion model
\begin{equation}\label{eq:pro}
\left\{\begin{array}{ll}
-\DIV(\diff[]\GRAD u) = f & \mbox{ in }\Omega,\\
u=0&\mbox{ on }\partial\Omega,
\end{array}\right.
\end{equation}
where $\Omega$ is a polytopal domain of $\Real^d$, $\diff[]:\Omega\to \Real^{d\times d}_{\rm sym}$ is a symmetric bounded uniformly coercive diffusion tensor, and $f\in L^2(\Omega)$. The solution to \eqref{eq:pro} is taken in the classical weak sense.

A review of historical or polytopal numerical methods on distorted meshes is out of this paper's scope. We however mention the recent works \cite{ABVW17,W19} about the VEM on anisotropic meshes, which present numerical results for a Poisson problem with internal layer, and derive approximation properties of the relevant interpolators. The novelty of our work, besides considering the HHO method instead of the VEM, is to establish complete error estimates (not just interpolator approximation properties) that take into account not only the distortion of the mesh, but also the high anisotropy of the diffusion tensor and the subtle interplay between these two features. The approach used here can be adapted to other methods, such as VEM, to yield error estimates that account for this interplay.

This paper is organised as follows. The concept of regular skewed mesh sequences, for which the error analysis will be carried out, is introduced in Section \ref{sec:meshes}; these meshes can have very elongated elements, provided that some local linear map transforms them into isotropic elements. The oblique elliptic projector is at the core of HHO schemes; its approximation properties on skewed elements are presented in Section \ref{sec:oblique.projector}, and are used in Section \ref{sec:analysis.hho} to perform the error analysis of HHO schemes on skewed meshes. This analysis is based on local transports of each skewed element $T$ into an isotropic element $\tT$; this transport identifies a new diffusion tensor on $\tT$, whose anisotropy properties dictate the contribution of $T$ to the global error estimate. The error estimate stated in Theorem \ref{th:error.est} therefore highlights how the diffusion anisotropy and the mesh skewness are combined in the multiplicative constants. This approach has the added advantage of leading to an error estimate that is as optimal as the standard error estimate for anisotropic diffusion models on regular (non-skewed) mesh sequences. In Section \ref{sec:numerics}, we perform a series of tests to evaluate the practical impact of high diffusion anisotropy and mesh skewness on the accuracy of HHO schemes. Some of the conclusions drawn from these tests are predicted by the error estimate but, overall, the HHO scheme is found to be more robust with respect to the diffusion anisotropy and mesh skewness than what the theoretical analysis seems to indicate. A conclusion is provided in Section \ref{sec:conclusion}.

\medskip

\noindent\textbf{Notations}. The Euclidean norm of a vector $\vec{\xi}\in\Real^d$ is denoted by $|\vec{\xi}|$. If $L:(\Real^d)^s\to\Real$ is an $s$-linear map, we define the norm of $L$ by
\[
\normml[s]{L}\coloneq\sup\{|L(\vec{\xi}_1,\ldots,\vec{\xi}_s)|\,:\,\vec{\xi}_i\in\Real^d\,,\;|\vec{\xi}_i|\le 1\,,\forall i=1,\ldots,s\}.
\]
For $X$ an open subset of $\Real^n$, $n\in\{d,d-1\}$, $(\cdot,\cdot)_X$ and $\norm[X]{\cdot}$ denote respectively the $L^2(X)$- or $L^2(X)^n$-inner product and norm. Letting $D^s v$ be the $s$-th order differential of $v$, the $H^s(X)$-seminorm of a function $v\in H^s(X)$ is $\seminorm[H^s(X)]{v}\coloneq \norm[X]{\normml[s]{D^s v}}$.

\section{Regular skewed mesh sequences}\label{sec:meshes}

Let us first briefly recall the definition of polytopal mesh, referring to \cite[Section 1.1]{hho-book} for details. A polytopal mesh of $\Omega$ is a couple $\Mh=(\Th,\Fh)$ where $\Th$ is a collection of disjoint polytopes $T$ ---the elements--- such that $\overline{\Omega}=\cup_{T\in\Th}\overline{T}$, and $\Fh$ is the set of mesh faces whose closures form a partition of $\cup_{T\in\Th}\partial T$, and such that each face is contained in one or two elements boundaries. Mesh faces can be different from the geometrical faces of the polytopes, the latter being possibly cut in two mesh faces in case of non-conforming mesh \cite[Fig. 1.2]{hho-book}. The diameter of a subset $X$ of $\Real^d$ is denoted by $h_X$. The index $h$ in $\Mh$ is the meshsize $h=\max_{T\in\Th} h_T$. For $T\in\Th$, we let $\Fh[T]$ be the set of faces $F\in\Fh$ such that $\partial T=\cup_{F\in\Fh[T]}\overline{F}$. The outer normal to $T$ on $F\in\Fh[T]$ is $\normal_{TF}$.
A matching simplicial mesh of $T\in\Th$ is a polytopal mesh of $T$ made of simplices and whose faces correspond to the geometrical simplicial faces.

We now define the concept of regular skewed mesh sequence, which allows for elements that become more and more stretched along the sequence, provided that each element can be linearly mapped onto an ``isotro\-pic'' element, that satisfies the regularity conditions of a standard regular mesh sequence \cite[Definition 1.9]{hho-book}.

\begin{definition}[Regular skewed mesh sequence] \label{def:reg.amesh}
Let $\mathcal H\subset (0,+\infty)$ be a countable set with $0$ as only accumulation point.
For each $h\in\mathcal H$, let $\Mh$ be a polytopal mesh and $\map[h]=(\map)_{T\in\Th}$ be a family of isomorphisms of $\Real^d$. The sequence $(\Mh,\map[h])_{h\in\mathcal H}$ is a \emph{regular skewed mesh sequence} if there exists $\varrho\in(0,1)$ such that, for all $h\in\mathcal H$ and all $T\in\Th$, the following properties hold:
\begin{enumerate}
\item\label{hT.htransT} Setting $\tT=\map(T)$, it holds $\varrho h_T\le h_{\tT}$ and  $\varrho h_{\tT}\le h_T$.
\item\label{tT.match} There is a matching simplicial mesh $(\fTh[\tT],\fFh[\tT])$ of $\tT$ such that, letting $\Fh[\tT]\coloneq\{\tF\coloneq\map(F)\,:\,F\in\Fh[T]\}$ be the set of faces of $\tT$, for any face $\sigma\in\fFh[\tT]$, either $\sigma\cap\partial\tT=\emptyset$ or there is $\tF\in\Fh[\tT]$ such that $\sigma\subset \tF$.
\item\label{transT.iso} For all $\tau\in \fTh[\tT]$, it holds $\varrho h_{\tT}\le h_\tau$ and $\varrho h_\tau\le r_\tau$, where $r_\tau$ is the inradius of $\tau$.
\end{enumerate}
\end{definition}


\begin{remark}[Comparison with \cite{W19}]
The notion of regular skewed mesh sequence is close to the notion of regular anisotropic mesh of \cite{W19}, in particular in the usage of maps from skewed elements to isotropic elements. A noticeable difference, however, is the requirement in \cite{W19} that two neighbouring elements $T,T'$ must have similar isotropy (that is, the corresponding mappings $\map[T],\map[T']$ must be close in a proper measure); this is due to the type of interpolators considered in \cite{W19}, which are adapted to VEM and therefore require to compute averaged values around each vertex. Such a requirement of similar isotropy for neighbouring elements is absent from Definition \ref{def:reg.amesh}, which is geared towards methods ---such as HHO--- whose interpolators are $L^2$-projections on cell and face polynomials; as a consequence, this definition allows for example for meshes with layers of very thin rectangles neighbouring layers of squares.
\end{remark}

In the rest of the paper, we consider a regular skewed mesh sequence $(\Mh,\map[h])_{h\in\mathcal H}$ with parameter $\varrho$, and we write $a\lesssim b$ if $a\le Cb$ with $C>0$ depending only on $\Omega$ and $\varrho$ and, when the inequality involves $H^s$ seminorms, also on the exponent $s$. We write $a\approx b$ if $a\lesssim b$ and $b\lesssim a$. We also make the following assumption.

\begin{assumption}[Piecewise constant diffusion tensor]
For all $h\in\mathcal{H}$, the diffusion tensor $\diff[]$ is piecewise constant on $\Th$. For any $T\in\Th$ we set $\diff=\diff[]_{|T}$.
\end{assumption}

Let $T$ be an element of one of the meshes $\Mh$. If $\vec{x}\in T$ we set $\trans{\vec{x}}=\map(\vec{x})\in\tT$. The gradient (resp. differential) with respect to $\trans{\vec{x}}$ is denoted by $\trans{\GRAD}$ (resp. $\trans{D}$). For $w\in L^2(T)$, the transport $\trans{w}\in L^2(\tT)$ of $w$ on $\tT$ is $\trans{w}(\trans{\vec{x}})=w(\vec{x})=w(\map^{-1}(\trans{\vec{x}}))$.
We also set $\Jmap=|{\det\map}|$, and define $\Jmap[T|F]$ as the absolute value of the determinant of the restriction $\map[T|F]:H_F\to H_{\tF}$, where $H_X$ denotes the hyperplane generated by $X=F$ or $\tF$; $\Jmap[T|F]$ can be computed using any pairs of orthonormal bases in $F$ and $\tF$. Letting $\map^t$ be the transpose of $\map$,
the relevant diffusion tensor on $\tT$ is:
\begin{equation}\label{eq:def.tdiff}
\tdiff=\map\diff\map^t.
\end{equation}
The maximal and minimal eigenvalues of $\tdiff$ are denoted by $\utdiff$ and $\ltdiff$.

\begin{lemma}[Transport relations]
\begin{enumerate}
\item\label{it1}\emph{Geometrical properties}. It holds $\normml[1]{\map^{-1}}\le \varrho^{-3}$ and,
for all $F\in\Fh[T]$, 
\begin{equation}\label{eq:phi.normal}
\map^t\normal_{\tT\tF}=\frac{\Jmap}{\Jmap[T|F]}\normal_{TF}.
\end{equation}

\item\label{it2}\emph{Transport of $L^2$-inner products and norms}. For all $w,z$ in $L^2(T)$ or $L^2(T)^d$, 
\begin{equation}
(w,z)_T=\Jmap^{-1} (\trans{w},\trans{z})_{\tT}\quad\mbox{ and }\quad
\norm[T]{w} = \Jmap^{-\nicefrac12}\norm[\tT]{\trans{w}}.
\label{eq:trans.norm}
\end{equation}
For all $F\in\Fh[T]$ and $w,z\in L^2(F)$,
\begin{equation}
(w,z)_F=\Jmap[T|F]^{-1}(\trans{w},\trans{z})_{\tF}\quad\mbox{ and }\quad
\norm[F]{w} = \Jmap[T|F]^{-\nicefrac12}\norm[\tF]{\trans{w}}.
\label{eq:trans.norm.F}
\end{equation}

\item\label{it3}\emph{Transport of derivatives}. For all $s\in\Natural$, $w\in H^s(T)$, $\vec{x}\in T$, it holds
\begin{equation}\label{eq:Ds.norm}
\normml[s]{\trans{D}^s\trans{w}(\trans{\vec{x}})}\lesssim \normml[s]{\trans{D^s w}(\trans{\vec{x}})}.
\end{equation}
For all $w,z\in H^1(T)$,
\begin{align}\label{eq:trans.grad}
&\trans{\GRAD w}(\trans{\vec{x}})=\GRAD w(\vec{x})=\map^t\trans{\GRAD}\trans{w}(\trans{\vec{x}})\qquad\forall \vec{x}\in T,\\
\label{eq:trans.grad.innerprod}
	&(\diff\nabla w,\nabla z)_T=\Jmap^{-1} (\tdiff\trans{\GRAD}\trans{w},\trans{\nabla}\trans{z})_{\tT}\,,\ \ 
\norm[T]{\diff^{\nicefrac12}\nabla w} = \Jmap^{-\nicefrac12}\norm[\tT]{\tdiff^{\nicefrac12}\trans{\nabla}\trans{w}}.
\end{align}

\end{enumerate}
\end{lemma}

\begin{proof}
\ref{it1}. We have $\map^{-1}(\tT)=T$. Since $\tT$ contains a ball of radius $\varrho^2 h_{\tT}$ (Point \ref{transT.iso} in Definition \ref{def:reg.amesh}) and $T$ has diameter $h_T\le\varrho^{-1}h_{\tT}$, we see that $\map^{-1}$ maps a ball of radius $\varrho^2 h_{\tT}$ into a ball of radius $\varrho^{-1}h_{\tT}$. Hence, $\normml[1]{\map^{-1}}\le (\varrho^{-1}h_{\tT})/(\varrho^2 h_{\tT})=\varrho^{-3}$.

Select two orthonormal bases $\mathcal B=(\mathcal B_F,\normal_{TF})$ and $\trans{\mathcal B}=(\trans{\mathcal B}_F,\normal_{\tT\tF})$ of $\Real^d$, where $\mathcal B_F$ is a basis of $H_F$ and $\trans{\mathcal B}_F$ is a basis of $H_{\tF}$. Since $\map(F)=\trans{F}$, the matrix of $\map$ in $(\mathcal B,\trans{\mathcal B})$ is written
\[
\left[
\begin{array}{cc} A & *\\
0 & \lambda
\end{array}\right],
\]
where $A$ is the matrix of $\map[T|F]$ in $(\mathcal B_F,\trans{\mathcal B}_F)$. In particular,
$\Jmap = |{\det A}|\lambda=\Jmap[T|F]\lambda$ and thus $\lambda=\Jmap/\Jmap[T|F]$. Transposing the matrix above gives the matrix of $\map^t$ in the orthonormal bases $(\trans{\mathcal B},\mathcal B)$. Since the last vector of $\trans{\mathcal B}$ (resp. $\mathcal B$) is $\normal_{\tT\tF}$ (resp. $\normal_{TF}$), reading the last column of this transposed matrix gives $\map^t \normal_{\tT\tF}=\lambda \normal_{TF}$ and proves \eqref{eq:phi.normal}.

\medskip

\ref{it2}. Simple changes of variables (in $T$ or $F$) establish \eqref{eq:trans.norm} and \eqref{eq:trans.norm.F}.

\medskip

\ref{it3}. Since $\trans{w}(\trans{\vec{x}})=w(\map^{-1}(\trans{\vec{x}}))$, an induction on $s$ shows that, for all $\vec{\xi}_1,\ldots,\vec{\xi}_s\in\Real^d$,
\begin{align}
\trans{D}^s\trans{w}(\trans{\vec{x}})(\vec{\xi}_1,\ldots,\vec{\xi}_s)={}&D^s w(\map^{-1}(\trans{\vec{x}}))(\map^{-1}(\vec{\xi}_1),\cdots,\map^{-1}(\vec{\xi}_s))\nonumber\\
={}&\trans{D^s w}(\trans{\vec{x}})(\map^{-1}(\vec{\xi}_1),\cdots,\map^{-1}(\vec{\xi}_s)).
\label{eq:Ds.trans}
\end{align}
We infer that $|\trans{D}^s\trans{w}(\trans{\vec{x}})(\vec{\xi}_1,\ldots,\vec{\xi}_s)|
\le\normml[s]{\trans{D^s w}(\trans{\vec{x}})}\,|\map^{-1}(\vec{\xi}_1)|\,\cdots\,|\map^{-1}(\vec{\xi}_s)|$.
By Point \ref{it1}, $|\map^{-1}(\vec{\xi}_i)|\lesssim|\vec{\xi}_i|$ for all $i=1,\ldots,s$, and the proof of  \eqref{eq:Ds.norm} is complete

The relation \eqref{eq:trans.grad} is obtained transposing \eqref{eq:Ds.trans} for $s=1$. The second relation in \eqref{eq:trans.grad.innerprod} follows from the first with $z=w$. To prove this first relation, apply \eqref{eq:trans.norm} to $\diff\GRAD w$ and $\GRAD z$ instead of $w$ and $z$, use the fact that $\diff$ is constant and invoke \eqref{eq:trans.grad} to write
\[
(\diff\GRAD w,\GRAD z)_T=\Jmap^{-1}(\diff\trans{\GRAD w},\trans{\GRAD z})_{\tT}
=\Jmap^{-1}(\map\diff\map^t\trans{\GRAD}\trans{w},\trans{\GRAD}\trans{z})_{\tT}.
\]
\end{proof}

\section{Oblique elliptic projector on skewed elements}\label{sec:oblique.projector}

Here, $T$ is a generic element of $\Mh$. Fix a polynomial degree $\ell\ge 0$ and recall the definition in \cite[Section 3.2.1]{hho-book} of the oblique elliptic projector $\dproj{\ell}:H^1(T)\to \Poly{\ell}(T)$: for all $v\in H^1(T)$,
\begin{align}
\label{def:dproj.grad}
(\diff\GRAD\dproj{\ell}v,\GRAD w)_T={}&(\diff\GRAD v,\GRAD w)_T\quad\forall w\in\Poly{\ell}(T),\\
\label{def:dproj.ave}
(\dproj{\ell}v,1)_T ={}& (v,1)_T.
\end{align}
The approximation properties of the oblique elliptic projector form an essential component of the analysis of HHO schemes for \eqref{eq:pro}. To establish these approximation properties, let us first describe how the elliptic projector is transported through $\map$.

\begin{lemma}[Transport of the elliptic projector]
Letting $\tdproj{\ell}$ be the oblique elliptic projector on $\tT$ for the tensor $\tdiff$ defined by \eqref{eq:def.tdiff}, it holds
\begin{equation}\label{eq:dproj.tdproj}
\widehat{\dproj{\ell}v}=\tdproj{\ell}\widehat{v} \qquad\forall v\in H^1(T).
\end{equation}
\end{lemma}

\begin{proof}
Take $w\in\Poly{\ell}(T)$ and write, using the definition \eqref{def:dproj.grad} of $\dproj{\ell}$, the transport relation \eqref{eq:trans.grad.innerprod} applied to $(v,w)$ instead of $(w,z)$, and the definition of $\tdproj{\ell}$ together with $\trans{w}\in\Poly{\ell}(\tT)$,
\begin{align}
(\diff\GRAD\dproj{\ell} v,\GRAD w)_T=(\diff\GRAD v,\GRAD w)_T
	={}& \Jmap^{-1}(\tdiff\trans{\GRAD}\trans{v},\trans{\GRAD}\trans{w})_{\tT}\nonumber\\
	={}&\Jmap^{-1}(\tdiff\trans{\GRAD}\tdproj{\ell}\trans{v},\trans{\GRAD}\trans{w})_{\tT}.
\label{eq:dproj.hproj.1}
\end{align}
On the other hand, \eqref{eq:trans.grad.innerprod} applied to $(\dproj{\ell} v,w)$ instead of $(w,z)$ gives
\[
(\diff\GRAD\dproj{\ell} v,\GRAD w)_T=\Jmap^{-1}(\tdiff\trans{\GRAD}\trans{\dproj{\ell}v},\trans{\GRAD}\trans{w})_{\tT}.
\]
Combining this relation with \eqref{eq:dproj.hproj.1} and using the fact that $\trans{w}$ is arbitrary in $\Poly{\ell}(\tT)$ yields $\trans{\GRAD}\tdproj{\ell}\trans{v}=\trans{\GRAD}\trans{\dproj{\ell}v}$. To prove \eqref{eq:dproj.tdproj} it remains to show that $\tdproj{\ell}\trans{v}$ and $\trans{\dproj{\ell}v}$ have the same average on $\tT$. This is done by using \eqref{eq:trans.norm} and \eqref{def:dproj.ave} (for both $\dproj{\ell}$ and $\tdproj{\ell}$) to write
$(\tdproj{\ell}\trans{v},1)_{\tT}=(\trans{v},1)_{\tT}=\Jmap(v,1)_T
=\Jmap(\dproj{\ell}v,1)_T=(\trans{\dproj{\ell}v},1)_{\tT}$.
\end{proof}

Let $\meas[n]{{\cdot}}$ be the $n$-dimensional Lebesgue measure. The following characteristic lengths will be used to state boundary approximation properties of $\dproj{\ell}$:
\begin{equation}\label{def:dTF}
	d_{TF}=\frac{\meas{T}}{\meas[d-1]{F}}\qquad\forall F\in\Fh[T].
\end{equation}
Using $\Jmap\meas{T}=\meas{\tT}$, $\Jmap[T|F]\meas[d-1]{F}=\meas[d-1]{\tF}$, and $\meas{\tT}\approx h_{\tF}\meas[d-1]{\tF}$ and $h_{\tF}\approx h_{\tT}$ (owing to the isotropy of $\tT$ and to \cite[Lemma 1.12]{hho-book}), we see that
\begin{equation}\label{eq:est.dTF}
d_{TF}\approx \frac{\Jmap[T|F]}{\Jmap} h_{\tF}\approx \frac{\Jmap[T|F]}{\Jmap} h_{\tT}.
\end{equation}

\begin{proposition}[Approximation properties of the elliptic projector on skewed elements]
For all $s\in\{1,\ldots,\ell+1\}$ and all $v\in H^s(T)$,
\begin{equation}\label{eq:dproj.approx.T}
\norm[T]{\diff^{\nicefrac12}\GRAD(v-\dproj{\ell}v)}\lesssim \utdiff^{\nicefrac12}h_T^{s-1} \seminorm[H^s(T)]{v}
\end{equation}
and, if $s\ge 2$, for all $F\in\Fh[T]$,
\begin{equation}\label{eq:dproj.approx.F}
d_{TF}^{\nicefrac12}\norm[F]{\diff^{\nicefrac12}\GRAD(v-\dproj{\ell}v)}\lesssim  \utdiff^{\nicefrac12} h_T^{s-1} \seminorm[H^s(T)]{v}.
\end{equation}
\end{proposition}

\begin{proof}
Since $\tT$ satisfies Points \ref{tT.match} and \ref{transT.iso} in Definition \ref{def:reg.amesh}, \cite[Theorem 3.3]{hho-book} yields
\begin{align}\label{eq:tdproj.approx}
\norm[\tT]{\tdiff^{\nicefrac12}\trans{\GRAD}(\trans{v}-\tdproj{\ell}\trans{v})}
\lesssim{}& \utdiff^{\nicefrac12}h_{\tT}^{s-1}\seminorm[H^s(\tT)]{\trans{v}},\\
\label{eq:tdproj.approx.F}
h_{\tT}^{\nicefrac12}\norm[\tF]{\tdiff^{\nicefrac12}\trans{\GRAD}(\trans{v}-\tdproj{\ell}\trans{v})}
\lesssim{}& \utdiff^{\nicefrac12}h_{\tT}^{s-1}\seminorm[H^s(\tT)]{\trans{v}}\quad\forall\tF\in\Fh[\tT]
\quad\mbox{ (if $s\ge 2$)}.
\end{align}
The volumetric \eqref{eq:dproj.approx.T} and trace \eqref{eq:dproj.approx.F} estimates are obtained transporting these estimates with \eqref{eq:dproj.tdproj}.
We start with the volumetric estimate. Using \eqref{eq:dproj.tdproj} and \eqref{eq:trans.grad.innerprod} we have
\[
\norm[\tT]{\tdiff^{\nicefrac12}\trans{\GRAD}(\trans{v}-\tdproj{\ell}\trans{v})}
=\norm[\tT]{\tdiff^{\nicefrac12}\trans{\GRAD}\trans{(v-\dproj{\ell}v)}}
=\Jmap^{\nicefrac12}\norm[T]{\diff^{\nicefrac12}\GRAD(v-\dproj{\ell}v)}.
\]
Hence, applying \eqref{eq:tdproj.approx} and using the estimate $h_{\tT}\lesssim h_T$ (see Point \ref{hT.htransT} in Definition \ref{def:reg.amesh}),
\begin{equation}\label{eq:dproj.approx.1}
\norm[T]{\diff^{\nicefrac12}\GRAD(v-\dproj{\ell}v)}\lesssim
\Jmap^{-\nicefrac12}\utdiff^{\nicefrac12}h_T^{s-1}\seminorm[H^s(\tT)]{\trans{v}}.
\end{equation}
By the definition of the $H^s$-seminorm, the relation \eqref{eq:Ds.norm} and the transport \eqref{eq:trans.norm} give
\begin{equation}\label{eq:est.Hs.transv}
\seminorm[H^s(\tT)]{\trans{v}}\lesssim \norm[\tT]{\normml[s]{\trans{D^s v}}}
\lesssim \Jmap^{\nicefrac12} \norm[T]{\normml[s]{D^s v}}=\Jmap^{\nicefrac12}\seminorm[H^s(T)]{v}.
\end{equation}
Plugged into \eqref{eq:dproj.approx.1}, this concludes the proof of \eqref{eq:dproj.approx.T}. 
We now turn to \eqref{eq:dproj.approx.F}. The transport relations \eqref{eq:dproj.tdproj}, \eqref{eq:trans.grad} and \eqref{eq:trans.norm.F} together with the definition \eqref{eq:def.tdiff} of $\tdiff$ yield
\[
\norm[\tF]{\tdiff^{\nicefrac12}\trans{\GRAD}(\trans{v}-\tdproj{\ell}\trans{v})}
=\norm[\tF]{\tdiff^{\nicefrac12}\trans{\GRAD}\trans{(v-\dproj{\ell}v)}}
=\Jmap[T|F]^{\nicefrac12}\norm[F]{\diff^{\nicefrac12}\GRAD(v-\dproj{\ell}v)}.
\]
Estimate \eqref{eq:dproj.approx.F} follows plugging this relation into \eqref{eq:tdproj.approx.F}, using \eqref{eq:est.Hs.transv} and recalling \eqref{eq:est.dTF} and that $h_{\tT}\lesssim h_T$. 
\end{proof}

\begin{remark}[Optimality of the approximation properties]\label{rem:optimality.dproj}
This proof shows that \eqref{eq:dproj.approx.T} and \eqref{eq:dproj.approx.F} come from the corresponding inequalities \eqref{eq:tdproj.approx} and \eqref{eq:tdproj.approx.F} for isotropic elements, and from \eqref{eq:est.Hs.transv}, itself derived from \eqref{eq:Ds.norm}. The latter inequality is optimal in the sense that, for any $\map$, there are functions $w$ for which it is an equality. Hence, the approximation properties \eqref{eq:dproj.approx.T} and \eqref{eq:dproj.approx.F} for skewed elements are as optimal as the corresponding approximation properties for isotropic elements.
\end{remark}

\section{Analysis of HHO schemes on skewed meshes}\label{sec:analysis.hho}

We briefly recall the construction of HHO schemes for \eqref{eq:pro} (referring to \cite[Chapter 3.1]{hho-book} for a comprehensive presentation), and establish key properties for proving error estimates on skewed meshes. In the following, $k\ge 0$ is a fixed polynomial degree.

\subsection{Local space and potential reconstruction}\label{sec:spaces}

For $T\in\Th$, the local space of unknowns is
\[
\UT\coloneq\{\uv=(v_T,(v_F)_{F\in\Fh[T]})\,:\,v_T\in\Poly{k}(T)\,,\;
v_F\in\Poly{k}(F)\quad\forall F\in\Fh\}.
\]
Setting $\sdiff[TF]=\diff\normal_{TF}\cdot\normal_{TF}$, this space is endowed with the seminorm
\begin{equation}\label{norm.H1KT}
\norm[1,{\diff[]},T]{\uv}\coloneq\left(\norm[T]{\diff^{\nicefrac12}\GRAD v_T}^2+\sum_{F\in\Fh[T]}\frac{\sdiff[TF]}{d_{TF}}\norm[F]{v_F-v_T}^2\right)^{\nicefrac12}\qquad\forall \uv\in\UT.
\end{equation}
For isotropic elements, this norm is usually defined using $h_F$ instead of $d_{TF}$, see \cite[Section 3.1.3.2]{hho-book}. The choice made in \eqref{norm.H1KT} ensures, for skewed elements, optimal estimates in terms of $\map$.
The potential reconstruction $\dpT:\UT\to\Poly{k+1}(T)$ is such that, for all $\uv[T]\in\UT$ and $w\in\Poly{k+1}(T)$,
\begin{align}
(\diff\GRAD\dpT,\GRAD w)_T={}&(\diff\GRAD v_T,\GRAD w)_T+\sum_{F\in\Fh[T]}(v_F-v_T,\diff\GRAD w\cdot\normal_{TF})_F,
\label{eq:dpT.1}
\\
\label{eq:dpT.2}
(\dpT\uv[T],1)_T={}&(v_T,1)_T.
\end{align}

\begin{lemma}[Transport of potential reconstruction]
It holds
\begin{equation}\label{eq:trans.dpT}
\trans{\dpT \uv[T]}=\dptT\trans{\uv[T]}\qquad\forall\uv[T]\in\UT,
\end{equation}
where $\trans{\uv[T]}=(\trans{v_T},(\trans{v_F})_{F\in\Fh[T]})\in\UtT$ is the transported $\uv[T]$, and $\dptT$ is the potential reconstruction on $\tT$ for the diffusion tensor $\tdiff$.
\end{lemma}

\begin{proof} 
For all $w\in\Poly{k+1}(T)$,
\begin{align}
(\tdiff\trans{\GRAD}{}&\trans{\dpT\uv[T]},\trans{\GRAD}\trans{w})_{\tT}=
	\Jmap(\diff\GRAD \dpT\uv[T],\GRAD w)_T\nonumber\\
	={}&
	\Jmap(\diff\GRAD v_T,\GRAD w)_T+\Jmap\sum_{F\in\Fh[T]}(v_F-v_T,\diff\GRAD w\cdot\normal_{TF})_F\nonumber\\
	={}&
	(\tdiff\trans{\GRAD}\trans{v_T},\trans{\GRAD}\trans{w})_{\tT}
	+\Jmap\sum_{F\in\Fh[T]}\Jmap[T|F]^{-1}(\trans{v_F}-\trans{v_T},\diff\trans{\GRAD w}\cdot\normal_{TF})_{\tF},
	\label{eq:dpT.dptT.1}
\end{align}
where we have used in this order the transport relation \eqref{eq:trans.grad.innerprod}, the definition \eqref{eq:dpT.1} of $\dpT$, the transport relations \eqref{eq:trans.grad.innerprod} and \eqref{eq:trans.norm.F}, and the fact that $\diff$ and $\normal_{TF}$ are constant.
Invoking \eqref{eq:trans.grad}, \eqref{eq:phi.normal} and \eqref{eq:def.tdiff}, we have 
\[
\diff\trans{\GRAD w}\cdot\normal_{TF}=\tdiff\trans{\GRAD}\trans{w}\cdot (\map^{-1})^t\normal_{TF}=
\frac{\Jmap[T|F]}{\Jmap}\tdiff\trans{\GRAD}\trans{w}\cdot\normal_{\tT\tF}
\]
and \eqref{eq:dpT.dptT.1} gives
\begin{align*}
(\tdiff\trans{\GRAD}\trans{\dpT\uv[T]},\trans{\GRAD}\trans{w})_{\tT}={}&
(\tdiff\trans{\GRAD}\trans{v_T},\trans{\GRAD}\trans{w})_{\tT}
	+\sum_{F\in\Fh[T]}(\trans{v_F}-\trans{v_T},\tdiff\trans{\GRAD}\trans{w}\cdot\normal_{\tT\tF})_{\tF}\\
={}&(\tdiff\trans{\GRAD}\dptT\trans{\uv[T]},\trans{\GRAD}\trans{w})_{\tT},
\end{align*}
the conclusion following from the definition of $\dptT$. Since $\trans{w}$ is arbitrary in $\Poly{k+1}(\tT)$, this proves that $\trans{\dpT\uv[T]}$ and $\dptT\trans{\uv[T]}$ have the same gradient. Using \eqref{eq:trans.norm} and \eqref{eq:dpT.2} we also see that they have same average on $\tT$, which concludes the proof of \eqref{eq:trans.dpT}.\end{proof}

\subsection{Local bilinear form}\label{sec:loc.bilinear}

The difference operators $\delta_{\diff[],T}^k:\UT\to\Poly{k}(T)$ and, for $F\in\Fh[T]$, $\delta_{\diff[],TF}^k:\UT\to\Poly{k}(F)$ are defined by: for all $\uv[T]\in\UT$,
\begin{equation}\label{eq:def.delta}
\delta_{\diff[],T}^k\uv[T]\coloneq \lproj[T]{k}(\dpT\uv[T]-v_T)\,,\quad
\delta_{\diff[],TF}^k\uv[T]=\lproj[F]{k}(\dpT\uv[T]-v_F)\qquad\forall F\in\Fh[T],
\end{equation}
where, for $X=T$ or $F$, $\lproj[X]{k}:L^2(X)\to\Poly{k}(X)$ is the $L^2(X)$-orthogonal projection. We note that, for any $w\in L^2(X)$,
\begin{equation}\label{eq:trans.proj}
\trans{\lproj[X]{k}w}=\lproj[\trans{X}]{k}\trans{w}.
\end{equation}
The local stabilisation bilinear form is given by: for all $\uu[T],\uv[T]\in\UT$,
\begin{equation}\label{eq:sT}
	\mathrm{s}_{\diff[],T}(\uu[T],\uv[T])\coloneq\sum_{F\in\Fh[T]}\frac{\sdiff[TF]}{d_{TF}}(\delta_{\diff[],TF}^k\uu[T]-\delta_{\diff[],T}^k\uu[T],\delta_{\diff[],TF}^k\uv[T]-\delta_{\diff[],T}^k\uv[T])_F.
\end{equation}
The local HHO bilinear form $\mathrm{a}_{\diff[],T}:\UT\times\UT\to\Real$ is then defined by:
\begin{equation}\label{eq:aT}
	\mathrm{a}_{\diff[],T}(\uu[T],\uv[T])\coloneq(\diff\GRAD\dpT\uu[T],\GRAD\dpT\uv[T])_T+\mathrm{s}_{\diff[],T}(\uu[T],\uv[T])\quad\forall \uu[T],\uv[T]\in\UT.
\end{equation}
In the right-hand side above, the first term is responsible for the consistency of the bilinear form, while the addition of the second term ensures the stability and boundedness property stated in the following proposition. Other choices of $\mathrm{s}_{\diff[],T}$ are possible \cite[Assumption 3.9]{hho-book}, and, on isotropic meshes, the factor $d_{TF}$ in this stabilisation bilinear form can be replaced by $h_F$.

\begin{proposition}[Stability and boundedness of {$\mathrm{a}_{\diff[],T}$}] It holds
\begin{equation}\label{eq:equiv.norms}
\mathrm{a}_{\diff[],T}(\uv[T],\uv[T])\approx \norm[1,{\diff[]},T]{\uv[T]}^2\qquad\forall \uv[T]\in\UT.
\end{equation}
\end{proposition}

\begin{proof}
\emph{Step 1: transport of seminorms}.
Let $\norm[1,{\diff[]},{\map[]},\tT]{{\cdot}}$ be defined on $\UtT$ by:
\[
\norm[1,{\diff[]},{\map[]},\tT]{\trans{\uv[T]}}^2\coloneq\norm[\tT]{\tdiff^{\nicefrac12}\trans{\GRAD}\trans{v_T}}^2+\sum_{\tF\in\Fh[\tT]}\frac{\stdiff[\tT\tF]}{h_{\tF}}\norm[\tF]{\trans{v_F}-\trans{v_T}}^2\qquad
\forall \trans{\uv[T]}\in\UtT,
\]
where $\stdiff[\tT\tF]\coloneq\tdiff\normal_{\tT\tF}\cdot\normal_{\tT\tF}$. If $\uv[T]\in\UT$, the transport relations \eqref{eq:trans.grad.innerprod} and \eqref{eq:trans.norm.F} yield
\begin{equation}
\norm[1,{\diff[]},T]{\uv[T]}^2=\Jmap^{-1}\norm[\tT]{\tdiff^{\nicefrac12}\trans{\GRAD}\trans{v_T}}^2+\sum_{F\in\Fh[T]}\frac{\sdiff[TF]}{d_{TF}}\Jmap[T|F]^{-1}\norm[\tF]{\trans{v_F}-\trans{v_T}}^2.
\label{eq:1T.trans}
\end{equation}
Starting from $\sdiff[TF]=\diff\normal_{TF}\cdot\normal_{TF}$, the relations \eqref{eq:phi.normal}, \eqref{eq:est.dTF} and \eqref{eq:def.tdiff} yield
\begin{equation}
\frac{\sdiff[TF]}{d_{TF}}\Jmap[T|F]^{-1}=\frac{\diff\frac{\Jmap[T|F]}{\Jmap}\map^t\normal_{\tT\tF}\cdot\frac{\Jmap[T|F]}{\Jmap}\map^t\normal_{\tT\tF}}{d_{TF}\Jmap[T|F]}
\approx\Jmap^{-1}\frac{\stdiff[\tT\tF]}{h_{\tF}}.
\label{eq:trans.KTF}
\end{equation}
Plugged into \eqref{eq:1T.trans}, this gives
\begin{equation}\label{eq:1T.trans.0}
\norm[1,{\diff[]},T]{\uv[T]}^2\approx \Jmap^{-1}\norm[1,{\diff[]},{\map[]},\tT]{\trans{\uv[T]}}^2.
\end{equation}

\emph{Step 2: transport of bilinear forms}.
Let $\mathrm{a}_{\diff[],\map[],\tT}:\UtT\times\UtT\to\Real$ be the standard local HHO bilinear form on $\tT$ for $\tdiff$:
\begin{align*}
	\mathrm{a}_{\diff[],\map[],\tT}(\trans{\uv[T]},\trans{\uw[T]})\coloneq{}&(\tdiff\trans{\GRAD}\dptT\trans{\uv[T]},\trans{\GRAD}\dptT\trans{\uw[T]})_{\tT}+
\mathrm{s}_{\diff[],\map[],\tT}(\trans{\uv[T]},\trans{\uw[T]}),\mbox{ where}\\
\mathrm{s}_{\diff[],\map[],\tT}(\trans{\uv[T]},\trans{\uw[T]})\coloneq{}&\sum_{\tF\in\Fh[\tT]}\frac{\stdiff[\tT\tF]}{h_{\tF}}(\delta_{\diff[],\map[],\tT\tF}^k\trans{\uv[T]}-\delta_{\diff[],\map[],\tT}^k\trans{\uv[T]},\delta_{\diff[],\map[],\tT\tF}^k\trans{\uw[T]}-\delta_{\diff[],\map[],\tT}^k\trans{\uw[T]})_{\tF}
\end{align*}
with difference operators $\delta_{\diff[],\map[],\tT}$ and $(\delta_{\diff[],\map[],\tT\tF})_{\tF\in\Fh[\tT]}$ defined on $\UtT$ in a similar way as in \eqref{eq:def.delta}, using $\dptT$ instead of $\dpT$.
Let $\uv[T]\in\UT$. Relations \eqref{eq:def.delta}, \eqref{eq:trans.dpT} and \eqref{eq:trans.proj} show that $\trans{\delta_{\diff[],T}^k\uv[T]}=\delta_{\diff[],\map[],\tT\tF}\trans{\uv[T]}$ and $\trans{\delta_{\diff[],TF}^k\uv[T]}=\delta_{\diff[],\map[],\tT\tF}\trans{\uv[T]}$. Hence, by \eqref{eq:trans.norm.F} and \eqref{eq:trans.KTF},
\begin{equation}\label{eq:trans.sT}
\mathrm{s}_{\diff[],T}(\uv[T],\uv[T])\approx \Jmap^{-1}\mathrm{s}_{\diff[],\map[],\tT}(\trans{\uv[T]},\trans{\uw[T]})
\end{equation}
and, recalling \eqref{eq:trans.grad.innerprod},
\[
(\diff\GRAD\dpT\uu[T],\GRAD\dpT\uv[T])_T=\Jmap^{-1}(\tdiff\trans{\GRAD}\dptT\trans{\uv[T]},\trans{\GRAD}\dptT\trans{\uv[T]})_T.
\]
This leads to
\begin{equation}
\mathrm{a}_{\diff[],T}(\uv[T],\uv[T])
\approx\Jmap^{-1}\mathrm{a}_{\diff[],\map[],\tT}(\trans{\uv[T]},\trans{\uv[T]}).
\label{eq:aT.atT}
\end{equation}

\emph{Step 3: conclusion}. Since $\tT$ is isotropic, \cite[Proposition 3.13]{hho-book} yields $\mathrm{a}_{\diff[],\map[],\tT}(\trans{\uv[T]},\trans{\uv[T]})\approx \norm[1,{\diff[]},{\map[]},\tT]{\trans{\uv[T]}}^2$. Using \eqref{eq:1T.trans.0} and \eqref{eq:aT.atT}, the proof of \eqref{eq:equiv.norms} is complete. \end{proof}

\subsection{HHO scheme and error estimate}\label{sec:estimates}

The global discrete space of unknowns is obtained patching local spaces and enforcing homogeneous Dirichlet boundary conditions:
\begin{align*}
\UhD\coloneq\{\uv[h]=({}&(v_T)_{T\in\Th},(v_F)_{F\in\Fh})\,:\,v_T\in\Poly{k}(T)\quad\forall T\in\Th\,,\\
&v_F\in\Poly{k}(F)\quad\forall F\in\Fh\,,\;v_F=0\quad\forall F\subset\partial\Omega\}.
\end{align*}
The restriction of $\uv[h]\in\UhD$ to an element $T$ is $\uv[T]=(v_T,(v_F)_{F\in\Fh[T]})\in\UT$. The interpolator $\Ih:H^1_0(\Omega)\to\UhD$ is such that, for $v\in H^1_0(\Omega)$,
\[
\Ih v\coloneq((\lproj[T]{k}v)_{T\in\Th},(\lproj[F]{k}v_{|F})_{F\in\Fh}).
\]
The local interpolator on $T\in\Th$ is $\IT:H^1(T)\to\UT$ such that, for $v\in H^1(T)$, $\IT v=(\lproj[T]{k}v,(\lproj[F]{k}v_{|F})_{F\in\Fh[T]})$.
The global HHO bilinear form $\mathrm{a}_{\diff[],h}:\UhD\times\UhD\to\Real$ is assembled from local contributions: for $\uv[h],\uw[h]\in\UhD$,
\[
\mathrm{a}_{\diff[],h}(\uu[h],\uv[h])\coloneq\sum_{T\in\Th}\mathrm{a}_{\diff[],T}(\uu[T],\uv[T]).
\]
This global bilinear form defines the energy norm such that, for $\uv[h]\in\UhD$,
\begin{equation}\label{eq:norm.aKh}
\norm[\mathrm{a},{\diff[]},h]{\uv[h]}\coloneq\mathrm{a}_{\diff[],h}(\uv[h],\uv[h])^{\nicefrac12}.
\end{equation}
The HHO scheme for \eqref{eq:pro} is written: find $\uu[h]\in\UhD$ such that
\begin{equation}\label{eq:hho}
\mathrm{a}_{\diff[],h}(\uu[h],\uv[h])=\sum_{T\in\Th}(f,v_T)_T\qquad\forall \uv[h]\in\UhD.
\end{equation}
This scheme is well-posed, and is a Finite Volume scheme in the sense that it has a flux formulation \cite[Lemma 3.17]{hho-book}. Our main result is the following theorem.

\begin{theorem}[Discrete energy error estimate for HHO schemes on skewed meshes]\label{th:error.est}
Assume that the weak solution $u\in H^1_0(\Omega)$ to \eqref{eq:pro} is such that, for some $r\in\{0,\ldots,k\}$, $u_{|T}\in H^{r+2}(T)$ for all $T\in\Th$. Let $\uu[h]\in\UhD$ be the solution to the HHO scheme \eqref{eq:hho}. Then, it holds
\begin{equation}\label{eq:error}
\norm[\mathrm{a},{\diff[]},h]{\Ih u-\uu[h]}\lesssim \left(\sum_{T\in\Th}\utdiff\ar{\diff[],\map[],\tT}h_T^{2(r+1)}\seminorm[H^{r+2}(T)]{u}^2\right)^{\nicefrac12},
\end{equation}
where $\ar{\diff[],\map[],\tT}$ is the anisotropy ratio of $\tdiff$, defined by $\ar{\diff[],\map[],\tT}\coloneq \frac{\utdiff}{\ltdiff}$.
\end{theorem}

\begin{remark}[Optimality of the error estimate]
Following Remark \ref{rem:optimality.dproj}, Estimate \eqref{eq:error} is as optimal with respect to the mesh skewness as the corresponding estimate \cite[Theorem 3.18]{hho-book}, for isotropic meshes, is optimal with respect to the diffusion tensor.
\end{remark}

\begin{remark}[Interplay between mesh skewness and diffusion anisotropy]\label{rem:interplay}
Assume for simplicity that $d=2$ and that, for any $T\in\Th$, there is an orthonormal basis in which
\begin{equation}\label{eq:form.diff.map}
\diff = \left[\begin{array}{cc} \lambda_T & 0 \\ 0 & 1 \end{array}\right]\quad\mbox{ and }\quad
\map=\left[\begin{array}{cc} a_T & 0 \\ 0 & b_T \end{array}\right].
\end{equation}
Then $\tdiff$ is diagonal with coefficients $a_T^2 \lambda_T$ and $b_T^2$, and \eqref{eq:error} leads to the estimate
\begin{multline}\label{eq:error.global}
\norm[\mathrm{a},{\diff[]},h]{\Ih u-\uu[h]}\\
\lesssim \max_{T\in\Th}\left[\max(a_T\lambda_T^{\nicefrac12},b_T)
\max\left(\frac{a_T\lambda_T^{\nicefrac12}}{b_T},\frac{b_T}{a_T\lambda_T^{\nicefrac12}}\right)\right]
h^{r+1}\seminorm[H^{r+2}(\Th)]{u},
\end{multline}
where $\seminorm[H^{r+2}(\Th)]{u}$ is the usual broken $H^{r+2}$-seminorm of $u$. The first term in the right-hand side of \eqref{eq:error.global} encodes the interaction between the skewness of the mesh elements and the local anisotropy of the diffusion tensor.
\end{remark}

\begin{proof}[Theorem \ref{th:error.est}]
Applying the 3rd Strang lemma \cite{Di-Pietro.Droniou:18}, we have
\begin{equation}\label{est.a}
\norm[\mathrm{a},{\diff[]},h]{\Ih u-\uu[h]}\le \sup_{\uv[h]\in\UhD,\,\norm[\mathrm{a},{\diff[]},h]{\uv[h]}\le 1}\mathcal E_{\diff[],h}(u;\uv[h]),
\end{equation}
where $\mathcal E_{\diff[],h}(u;\uv[h])\coloneq\sum_{T\in\Th}(f,v_T)_T - \mathrm{a}_{\diff[],h}(\Ih u,\uv[h])$. The following relation is established in the proof of \cite[Lemma 3.15]{hho-book}:
\begin{equation}
\begin{aligned}
	\mathcal E_{\diff[],h}(u;\uv[h])={}&\sum_{T\in\Th}\sum_{F\in\Fh[T]}(\diff\GRAD(u-\dproj{k+1}u)\cdot\normal_{TF},v_F-v_T)_F\\
&-\sum_{T\in\Th}\mathrm{s}_{\diff[],T}(\IT u,\uv[T])
=\term_1+\term_2.
\end{aligned}
\label{est.E}
\end{equation}
Let $\uv[h]\in\UhD$ be such that $\norm[\mathrm{a},{\diff[]},h]{\uv[h]}\le 1$. Writing $\diff\GRAD(u-\dproj{k+1}u)\cdot\normal_{TF}=\diff^{\nicefrac12}\GRAD(u-\dproj{k+1}u)\cdot\diff^{\nicefrac12}\normal_{TF}$, using Cauchy--Schwarz inequalities and $|\diff^{\nicefrac12}\normal_{TF}|=\sdiff[TF]^{\nicefrac12}$, we have
\begin{align}
|\term_1|\le{}& \sum_{T\in\Th}\sum_{F\in\Fh[T]}\sdiff[TF]^{\nicefrac12}\norm[F]{\diff^{\nicefrac12}\GRAD(u-\dproj{k+1}u)}\norm[F]{v_F-v_T}\nonumber\\
\le{}&\left(\sum_{T\in\Th}\sum_{F\in\Fh[T]}d_{TF}\norm[F]{\diff^{\nicefrac12}\GRAD(u-\dproj{k+1}u)}^2\right)^{\nicefrac12}
\left(\sum_{T\in\Th}\sum_{F\in\Fh[T]}\frac{\sdiff[TF]}{d_{TF}}\norm[F]{v_F-v_T}^2\right)^{\nicefrac12}\nonumber\\
\lesssim{}&\left(\sum_{T\in\Th}\utdiff h_T^{2(r+1)}\seminorm[H^{r+2}(T)]{u}^2
\right)^{\nicefrac12},
\label{est:T1}
\end{align}
where we have used \eqref{eq:dproj.approx.F} (with $\ell=k+1$ and $s=r+2$) and the norm equivalence \eqref{eq:equiv.norms} to write $\sum_{T\in\Th}\sum_{F\in\Fh[T]}\frac{\sdiff[TF]}{d_{TF}}\norm[F]{v_F-v_T}^2\lesssim \sum_{T\in\Th}\mathrm{a}_{\diff[],T}(\uv[T],\uv[T])=\mathrm{a}_{\diff[],h}(\uv[h],\uv[h])\le 1$.
To estimate $\term_2$, we also use Cauchy--Schwarz inequalities, the bound $\sum_{T\in\Th}\mathrm{s}_{\diff[],T}(\uv[T],\uv[T])\le \sum_{T\in\Th}\mathrm{a}_{\diff[],T}(\uv[T],\uv[T])\le 1$
and the transport relation \eqref{eq:trans.sT} to write
\[
|\term_2|\le\left(\sum_{T\in\Th}\mathrm{s}_{\diff[],T}(\IT u,\IT u)\right)^{\nicefrac12}\lesssim\left(\sum_{T\in\Th}\Jmap^{-1}\mathrm{s}_{\diff[],\map[],\tT}(\trans{\IT u},\trans{\IT u})\right)^{\nicefrac12}.
\]
Since $\tT$ is isotropic and $\trans{\IT u}=\ItT\trans{u}$ (owing to \eqref{eq:trans.proj}), the consistency properties \cite[Lemma 3.10]{hho-book} of $\mathrm{s}_{\diff[],\map[],\tT}$ and the relations $h_{\tT}\lesssim h_T$ and \eqref{eq:est.Hs.transv} yield
\begin{align*}
|\term_2|\lesssim{}&\left(\sum_{T\in\Th}\Jmap^{-1}\utdiff[\tT]\ar{\diff[],\map[],\tT}h_{\tT}^{2(r+1)}\seminorm[H^{r+2}(\tT)]{\trans{u}}^2\right)^{\nicefrac12}\\
\lesssim{}&
\left(\sum_{T\in\Th}\utdiff[\tT]\ar{\diff[],\map[],\tT}h_{T}^{2(r+1)}\seminorm[H^{r+2}(T)]{u}^2\right)^{\nicefrac12}.
\end{align*}
Plug this estimate and \eqref{est:T1} into \eqref{est.E}, use $\ar{\diff[],\map[],\tT}\ge 1$ and recall \eqref{est.a} to conclude. \end{proof}

\section{Numerical evaluation of the effects of diffusion ani\-sotropy and mesh skewness}\label{sec:numerics}

We provide here a series of numerical results, on the domain $\Omega=(0,1)^2$ (and with non-homogeneous Dirichlet boundary conditions ---see \cite[Section 2.4]{hho-book} for the adaptation of the scheme \eqref{eq:hho} to this case), to assess the practical optimality of the error estimate \eqref{eq:error} and its consequence \eqref{eq:error.global} in cases of highly anisotropic diffusion tensor and/or skewed mesh families.
The accuracy of the HHO scheme is measured through the following two relative errors:
\[
E_{\mathrm{a},\diff[],h}\coloneq \frac{\norm[\mathrm{a},{\diff[]},h]{\Ih u-\uu[h]}}{\norm[\mathrm{a},{\diff[]},h]{\Ih u}}\quad\mbox{ and }\quad
E_{1,h}\coloneq \frac{\norm[1,h]{\Ih u-\uu[h]}}{\norm[1,h]{\Ih u}}\,,
\]
where $\norm[\mathrm{a},{\diff[]},h]{{\cdot}}$ is defined by \eqref{eq:norm.aKh}, and $\norm[1,h]{{\cdot}}$ is the diffusion-independent discrete $H^1$-norm obtained adding together the local seminorms \eqref{norm.H1KT} with $\diff[]=\Id$, that is:
\begin{equation}\label{norm.H1}
\norm[1,h]{\uv[h]}\coloneq\left(\sum_{T\in\Th}\Big[\norm[T]{\GRAD v_T}^2+\sum_{F\in\Fh[T]}d_{TF}^{-1}\norm[F]{v_F-v_T}^2\Big]\right)^{\nicefrac12}.
\end{equation}

The numerical tests have been performed using the code ``\texttt{HHO-Diffusion}'' in the C++ open source library \texttt{HArDCore} \cite{HArDCore}. This library provides generic tools for implementing 2D and 3D numerical methods with unknowns made of polynomials on the edges/faces and cells of the mesh; it also naturally handles generic polygonal and polyhedral meshes. The variety of possible tests to assess the practical efficiency of the scheme \eqref{eq:hho} with anisotropic diffusion/skewed meshes is infinite, given the numerous possible parameters (polynomial degrees $k$, diffusion tensors, exact solutions, type of meshes, etc). We only report a few relevant results here, but all the meshes and data used in the tests below are available in \texttt{HArDCore} for the interested reader to run additional tests.

\subsection{Test A: anisotropic diffusion tensor}

This test focuses on the effect of an anisotropic and heterogeneous diffusion tensor. For $\lambda\in\{10^{-6},1,10^6\}$, we consider the tensor
\[
\diff[](x,y)=
\left[\begin{array}{cc}\lambda&0\\0&1\end{array}\right]\quad\mbox{ if $y<0.5$},\qquad
\diff[](x,y)=\Id\quad\mbox{ if $y\ge 0.5$},
\]
and fix the exact solution $u(x,y)=\cos(\pi x)\cos(\pi y)$; the source term and boundary conditions are computed from this solution. Since $(\partial_x u)_{|y=0.5}=0$, we still have $\DIV(\diff[]\GRAD u)\in L^2(\Omega)$ despite the discontinuity of $\diff[]$ along $y=0.5$. We consider a family of locally refined meshes from the FVCA5 benchmark \cite{Herbin.Hubert:08} (see Fig. \ref{fig:mesh.testA}), for which the setting of Remark \ref{rem:interplay} holds with $\lambda_T=\lambda$ and $a_T=b_T=1$; the estimate \eqref{eq:error.global} therefore predicts a dependency of the energy error on $\max(\lambda,\lambda^{-\frac12})$. The results for $k=1,3$ are presented in Fig. \ref{fig:A.locref}; tests with other polynomial degrees present the same trend. The energy error $E_{\mathrm{a},\diff[],h}$ appears to depend much less on the anisotropy ratio than predicted; the error $E_{1,h}$ shows a more pronounced dependency on the tensor anisotropy, especially for low degrees where a factor of about 30 is seen on the finest mesh between $\lambda=1$ and $\lambda=10^{-6}, 10^6$.

\begin{figure}
\begin{center}
\begin{tabular}{cc}
\includegraphics[width=0.35\linewidth]{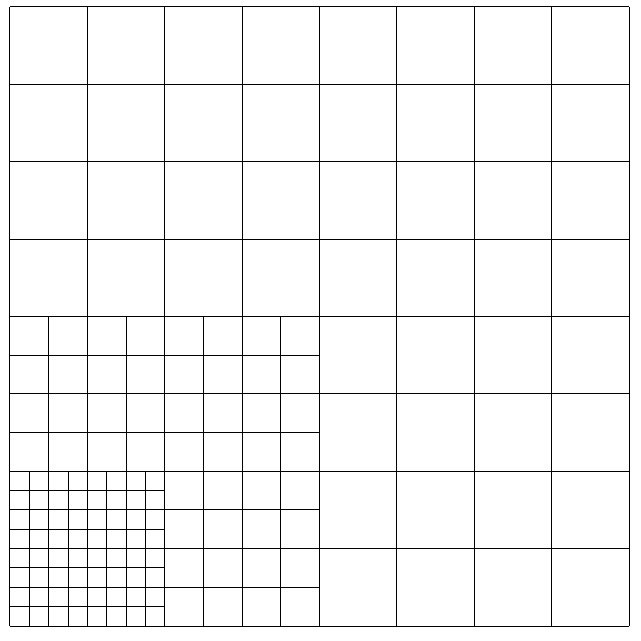} & 
\includegraphics[width=0.35\linewidth]{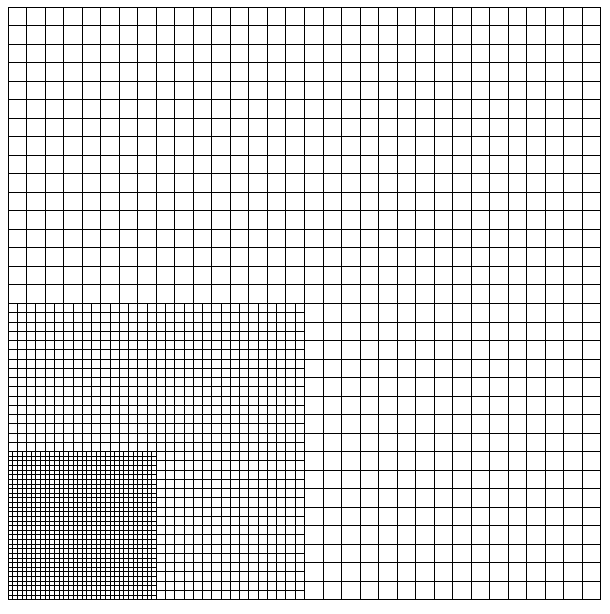} 
\end{tabular}
\caption{Two members of the family of meshes used in Test A \label{fig:mesh.testA}.}
\end{center}
\end{figure}

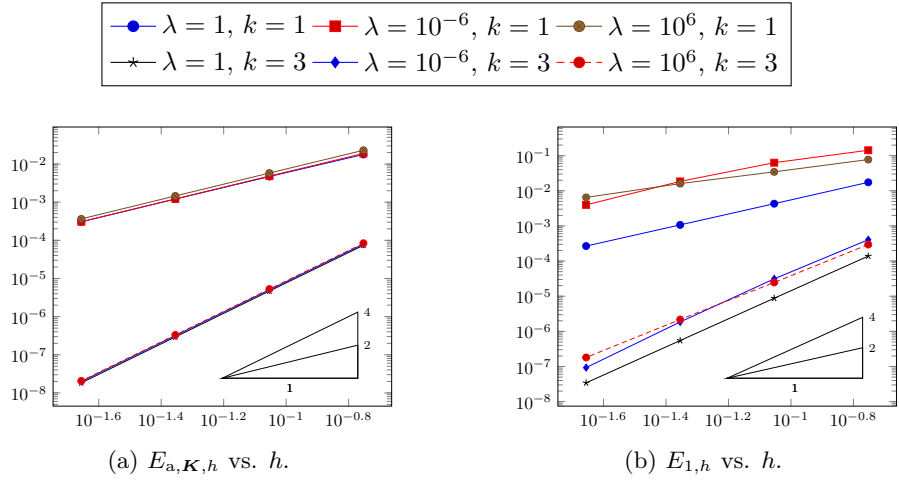
\begin{figure}\centering
  \ref{conv.testA.2}
  \vspace{0.5cm}\\
  \begin{minipage}[b]{0.45\linewidth}
    \centering
    \begin{tikzpicture}[scale=0.65]
      \begin{loglogaxis}[ legend columns=3, legend to name=conv.testA.2 ]
        \addplot table[x=meshsize,y=EnergyError] {dat/testA/locref/k1.lambda.1.dat};
        \addplot table[x=meshsize,y=EnergyError] {dat/testA/locref/k1.lambda.1e-6.dat};
        \addplot table[x=meshsize,y=EnergyError] {dat/testA/locref/k1.lambda.1e6.dat};
        \addplot table[x=meshsize,y=EnergyError] {dat/testA/locref/k3.lambda.1.dat};
        \addplot table[x=meshsize,y=EnergyError] {dat/testA/locref/k3.lambda.1e-6.dat};
        \addplot table[x=meshsize,y=EnergyError] {dat/testA/locref/k3.lambda.1e6.dat};
        \logLogSlopeTriangle{0.90}{0.4}{0.1}{2}{black};
        \logLogSlopeTriangle{0.90}{0.4}{0.1}{4}{black};
        \addlegendentry{$\lambda=1$, $k=1$}
				\addlegendentry{$\lambda=10^{-6}$, $k=1$}
				\addlegendentry{$\lambda=10^{6}$, $k=1$}
				\addlegendentry{$\lambda=1$, $k=3$}
				\addlegendentry{$\lambda=10^{-6}$, $k=3$}
				\addlegendentry{$\lambda=10^{6}$, $k=3$}
      \end{loglogaxis}
    \end{tikzpicture}
    \subcaption{$E_{\mathrm{a},{\diff[]},h}$ vs. $h$.\label{fig:A.locref.energy}}
  \end{minipage}
  \begin{minipage}[b]{0.45\linewidth}
    \centering
    \begin{tikzpicture}[scale=0.65]
      \begin{loglogaxis}[ legend columns=3, legend to name=conv.testA.2 ]
        \addplot table[x=meshsize,y=H1error] {dat/testA/locref/k1.lambda.1.dat};
        \addplot table[x=meshsize,y=H1error] {dat/testA/locref/k1.lambda.1e-6.dat};
        \addplot table[x=meshsize,y=H1error] {dat/testA/locref/k1.lambda.1e6.dat};
        \addplot table[x=meshsize,y=H1error] {dat/testA/locref/k3.lambda.1.dat};
        \addplot table[x=meshsize,y=H1error] {dat/testA/locref/k3.lambda.1e-6.dat};
        \addplot table[x=meshsize,y=H1error] {dat/testA/locref/k3.lambda.1e6.dat};
        \logLogSlopeTriangle{0.90}{0.4}{0.1}{2}{black};
        \logLogSlopeTriangle{0.90}{0.4}{0.1}{4}{black};
        \addlegendentry{$\lambda=1$, $k=1$}
				\addlegendentry{$\lambda=10^{-6}$, $k=1$}
				\addlegendentry{$\lambda=10^{6}$, $k=1$}
				\addlegendentry{$\lambda=1$, $k=3$}
				\addlegendentry{$\lambda=10^{-6}$, $k=3$}
				\addlegendentry{$\lambda=10^{6}$, $k=3$}
      \end{loglogaxis}
    \end{tikzpicture}
    \subcaption{$E_{1,h}$ vs. $h$.\label{fig:A.locref.H1}}
  \end{minipage}
\caption{Errors vs. meshsize for Test A. Slopes = rates expected from \eqref{eq:error}.\label{fig:A.locref}}
\end{figure}

\subsection{Test B: skewed mesh}

In this test, we study the impact of the mesh skewness. We take $\diff[]=\Id$ and the exact solution $u(x,y)=\cos(\pi x)\cos(\pi y)$. The meshes are (mostly) hexagonal, and more and more skewed as their size decreases (see Fig. \ref{fig:mesh.hexa}). The results in Fig. \ref{fig:B.hexa} show a clear loss of rate of convergence, compared to the expected rate for isotropic meshes.

\begin{figure}
\begin{center}
\begin{tabular}{cc}
\includegraphics[width=0.35\linewidth]{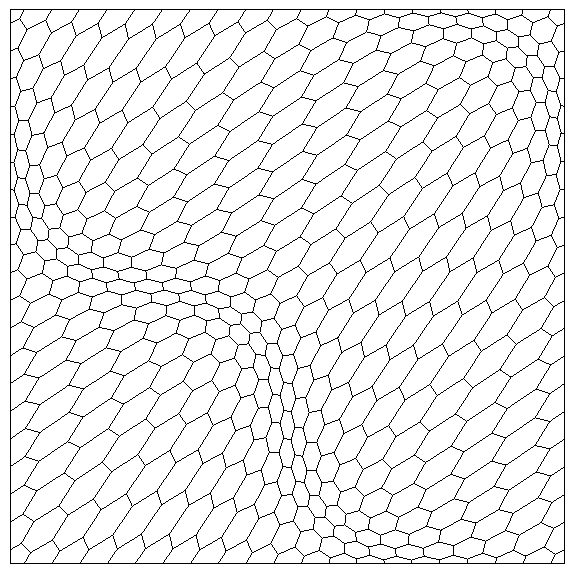}&
\includegraphics[width=0.35\linewidth]{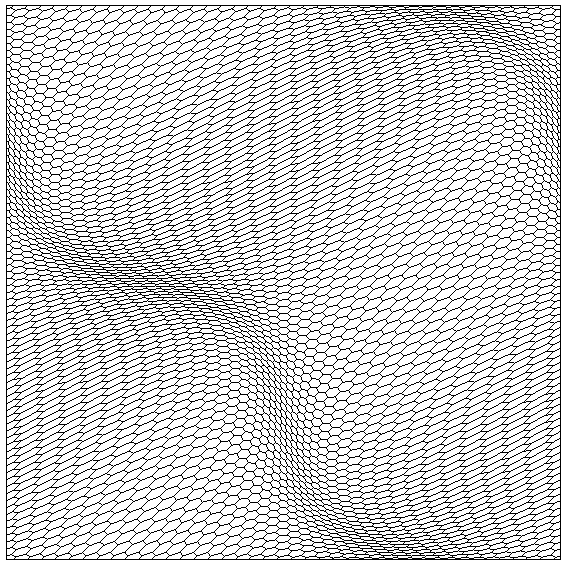}
\end{tabular}
\caption{First two meshes in the skewed family used in Test B. \label{fig:mesh.hexa}}
\end{center}
\end{figure}

\begin{figure}\centering
  \ref{conv.testB}
  \vspace{0.5cm}\\
  \begin{minipage}[b]{0.45\linewidth}
    \centering
    \begin{tikzpicture}[scale=0.65]
      \begin{loglogaxis}[ legend columns=-1, legend to name=conv.testB ]
        \addplot table[x=meshsize,y=EnergyError] {dat/testB/hexa/k0.hexa_anisotropic.dat};
        \addplot table[x=meshsize,y=EnergyError] {dat/testB/hexa/k1.hexa_anisotropic.dat};
        \addplot table[x=meshsize,y=EnergyError] {dat/testB/hexa/k2.hexa_anisotropic.dat};
        \addplot table[x=meshsize,y=EnergyError] {dat/testB/hexa/k3.hexa_anisotropic.dat};
        \logLogSlopeTriangle{0.90}{0.4}{0.1}{1}{black};
        \logLogSlopeTriangle{0.90}{0.4}{0.1}{2}{black};
        \logLogSlopeTriangle{0.90}{0.4}{0.1}{3}{black};
        \logLogSlopeTriangle{0.90}{0.4}{0.1}{4}{black};
        \legend{$k=0$,$k=1$,$k=2$,$k=3$};
      \end{loglogaxis}
    \end{tikzpicture}
    \subcaption{$E_{\mathrm{a},{\diff[]},h}$ vs. $h$.\label{fig:B.hexa_aniso}}
  \end{minipage}
  \begin{minipage}[b]{0.45\linewidth}
    \centering
    \begin{tikzpicture}[scale=0.65]
      \begin{loglogaxis}[ legend columns=-1, legend to name=conv.testB ]
        \addplot table[x=meshsize,y=H1error] {dat/testB/hexa/k0.hexa_anisotropic.dat};
        \addplot table[x=meshsize,y=H1error] {dat/testB/hexa/k1.hexa_anisotropic.dat};
        \addplot table[x=meshsize,y=H1error] {dat/testB/hexa/k2.hexa_anisotropic.dat};
        \addplot table[x=meshsize,y=H1error] {dat/testB/hexa/k3.hexa_anisotropic.dat};
        \logLogSlopeTriangle{0.90}{0.4}{0.1}{1}{black};
        \logLogSlopeTriangle{0.90}{0.4}{0.1}{2}{black};
        \logLogSlopeTriangle{0.90}{0.4}{0.1}{3}{black};
        \logLogSlopeTriangle{0.90}{0.4}{0.1}{4}{black};
        \legend{$k=0$,$k=1$,$k=2$,$k=3$};
      \end{loglogaxis}
    \end{tikzpicture}
    \subcaption{$E_{1,h}$ vs. $h$.\label{fig:B.hexa_aniso.H1}}
  \end{minipage}
\caption{Errors vs. meshsize for Test B. The slopes indicate the expected rates of convergence $h^{k+1}$, disregarding the effects of the mesh skewness.
\label{fig:B.hexa}}
\end{figure}
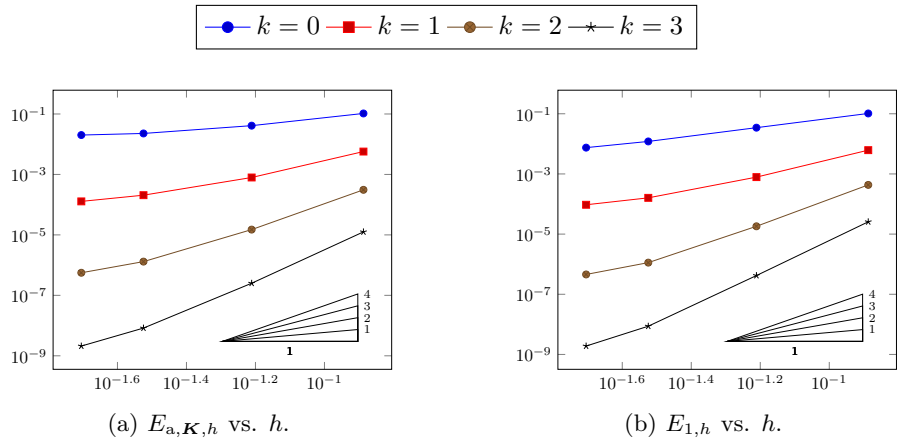

To estimate more precisely the effect of mesh skewness, we introduce the \emph{flatness} factor defined as $\fl_h\coloneq \max_{T\in\Th}\fl_T$ with $\fl_T\coloneq\frac{h_T}{\rho_T}$, where $\rho_T$ is the radius of the largest ball centred at the centre of mass of $T$ and contained in $T$. The skewness of the considered meshes comes from the large flatness factors $\fl_T$ of some elements $T$. It is easy to convince oneself that this setting is compatible with Remark \ref{rem:interplay} with $\lambda=1$, $a_T=\fl_T$ and $b_T=1$. As a consequence, \eqref{eq:error.global} predicts an upper bound
\begin{equation}\label{eq:rate.rd}
E_{\mathrm{a},\diff[],h}\lesssim \fl_h^2 h^{k+1}\seminorm[H^{k+2}(\Th)]{u}.
\end{equation}
To evaluate the accuracy of this estimate with respect to the mesh flatness, for each error $E_h\in\{E_{\mathrm{a},\diff[],h},E_{1,h}\}$ we provide in Table \ref{tab:rates.flatness} an evaluation of the rates of growth of $E_h/h^{k+1}$ with respect to $\fl_h$. Estimate \eqref{eq:rate.rd} tells us that, at least for the energy error, this rate should be at a maximum of 2.
As can be seen in Table \ref{tab:rates.flatness}, the actual rates are much smaller than 2, and both errors are less sensitive to the mesh flatness than \eqref{eq:rate.rd} predicts; the diffusion-independent norm $E_{1,h}$ is the least sensitive of both.

Table \ref{tab:rates.flatness} also reports the condition numbers (CN) in 1-norm for the statically condensed system. For regular mesh sequences, CNs of HHO systems grow as $h^{-2}$. Here, the growth is in $h^{-4}$ (but the CNs do not depend much on $k$). The additional power of $2$ could come from a factor $\fl^2$ (since, here, $\fl\sim 1/h$). Further analysis and tests are however necessary to reach a definitive conclusion, and it should also be noted that the meshes considered here contain a large portion of skewed elements; the condition numbers could be reduced for meshes with a smaller portion of distorted cells.

\begin{table}
\begin{center}
	\resizebox{\textwidth}{!}{
\begin{tabular}{c@{\hspace*{1em}}c}
\begin{tabular}{|c|c|c||c|c||c|c|}
\hline
$h$ & $\fl_h$  & CN &  $\frac{E_{\mathrm{a},{\diff[]},h}}{h^{k+1}}$ & rate & 
 $\frac{E_{1,h}}{h^{k+1}}$ & rate \\
\hline
0.13 & 10 & 875 & 8e-01 & -- & 7.9e-01 & -- \\
0.06 & 22 & 1.8e+04 &  6.7e-01 & -0.2 & 5.6e-01 & -0.5 \\ 
0.03 & 46 & 2.6e+05 &  7.6e-01 & 0.2 & 4.0e-01 & -0.4 \\ 
0.02 & 70 & 1.3e+06 &  1e+00 & 0.7 & 3.8e-01 & -0.1 \\
\hline
\end{tabular}
&
\begin{tabular}{|c|c|c||c|c||c|c|}
\hline
$h$ & $\fl_h$ & CN &  $\frac{E_{\mathrm{a},{\diff[]},h}}{h^{k+1}}$ & rate & 
 $\frac{E_{1,h}}{h^{k+1}}$ & rate \\
\hline
0.13 & 10 & 1.7e+03 & 3.4e-01 & -- & 3.7e-01 & -- \\
0.06 & 22 & 3.1e+04 &  2.1e-01 & -0.6 & 2.1e-01 & -0.7 \\ 
0.03 & 46 & 5.0e+05 &  2.3e-01 & 0.1 & 1.8e-01 & -0.2 \\ 
0.02 & 70 & 2.6e+06 &  3.3e-01 & 0.8 & 2.4e-01 & 0.7 \\
\hline
\end{tabular}\\[2.8em]
$k=0$ & $k=1$ \\[1em]
\begin{tabular}{|c|c|c||c|c||c|c|}
\hline
$h$ & $\fl_h$ & CN &  $\frac{E_{\mathrm{a},{\diff[]},h}}{h^{k+1}}$ & rate & 
 $\frac{E_{1,h}}{h^{k+1}}$ & rate \\
\hline
0.13 & 10 & 2.7e+03 & 1.4e-01 & -- & 2.0e-01 & -- \\
0.06 & 22 & 4.3e+04 &  6.4e-02 & -1 & 7.8e-02 & -1.2 \\ 
0.03 & 46 & 7.7e+05 &  4.9e-02 & -0.4 & 4.2e-02 & -0.8 \\ 
0.02 & 70 & 4.0e+06 &  7.3e-02 & 0.9 & 5.9e-02 & 0.8 \\ 
\hline
\end{tabular}
&
\begin{tabular}{|c|c|c||c|c||c|c|}
\hline
$h$ & $\fl_h$ & CN &  $\frac{E_{\mathrm{a},{\diff[]},h}}{h^{k+1}}$ & rate & 
 $\frac{E_{1,h}}{h^{k+1}}$ & rate \\
\hline
0.13 & 10 & 3.9e+03 & 4.4e-02 & -- & 9.1e-02 & -- \\
0.06 & 22 & 5.6e+04 &  1.8e-02 & -1.2 & 2.9e-02 & -1.5 \\ 
0.03 & 46 & 1.1e+06 &  1.0e-02 & -0.7 & 1.1e-02 & -1.3 \\ 
0.02 & 70 & 5.6e+06 &  1.4e-02 & 0.7 & 1.3e-02 & 0.3 \\
\hline
\end{tabular}\\[2.8em]
$k=2$ & $k=3$
\end{tabular}
}
\caption{Rates of convergence of the errors with respect to the mesh flatness, Test B.
\label{tab:rates.flatness}}
\end{center}
\end{table}

\subsection{Test C}

We assess here the interplay between mesh skewness and diffusion anisotropy, taking $\diff[](x,y)=\mathrm{diag}(10^6,1)$, and $u(x,y)=\cos(\pi x)\cos(\pi y)$ as before. We consider two families of meshes: regular hexagonal, and skewed hexagonal with flatness factor multiplied by two from one mesh member to the next; see Fig. \ref{fig:hexa_meshes}.

\begin{figure}
\begin{center}
\begin{tabular}{c@{\qquad}c}
\adjincludegraphics[width=0.3\linewidth, trim={0 {.9\width} {.8\width} 0}, clip]{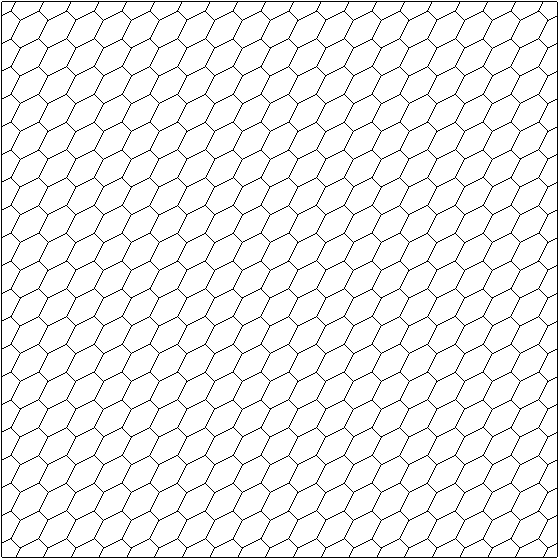}&
\adjincludegraphics[width=0.3\linewidth, trim={0 {.9\width} {.8\width} 0}, clip]{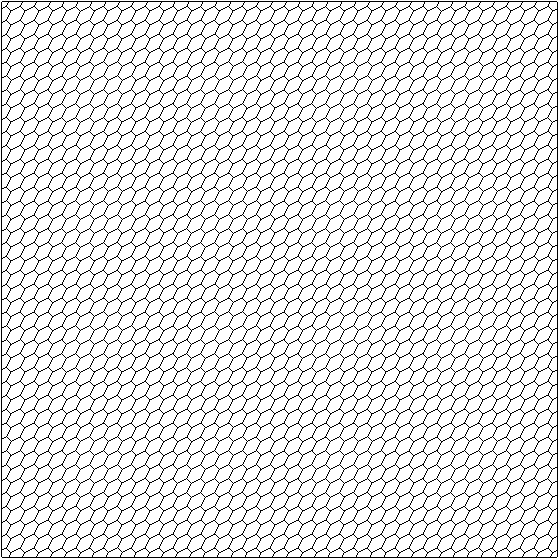}\\[1em]
\adjincludegraphics[width=0.3\linewidth, trim={0 {.9\width} {.8\width} 0}, clip]{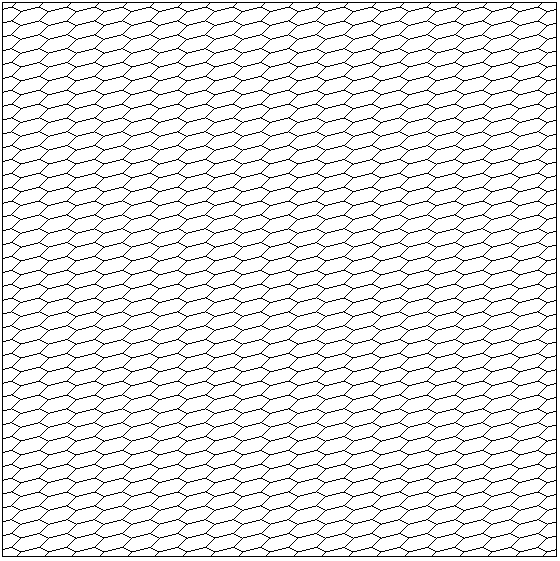}&
\adjincludegraphics[width=0.3\linewidth, trim={0 {.9\width} {.8\width} 0}, clip]{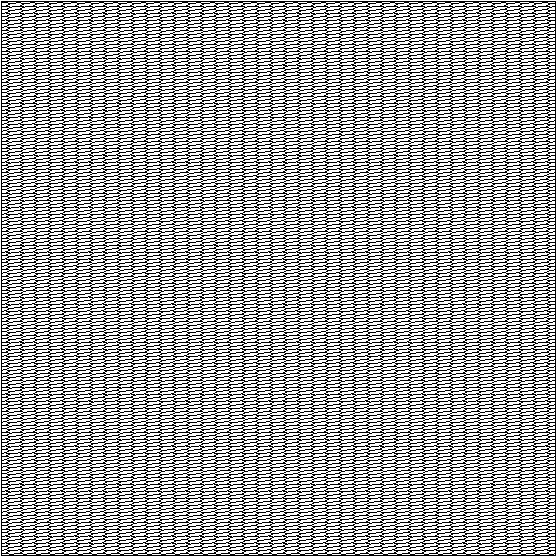}
\end{tabular}
\caption{Upper left corner of the meshes in Test C: regular hexagonal (top); skewed hexagonal (bottom). \label{fig:hexa_meshes}}
\end{center}
\end{figure}

The setting of Remark \ref{rem:interplay} is valid with $(a_T,b_T)=(1,\fl_h)$ with $\fl_h\le 10^3$ for the considered meshes; \eqref{eq:error.global} thus predicts a bound $E_{\mathrm{a},\diff[],h}\lesssim 10^6 \fl_h^{-1}h^{k+1}\seminorm[H^{k+2}(\Omega)]{u}$.
For the skewed meshes, we have $\fl_h\sim 1/h$ and we therefore expect a better rate of convergence than the usual $h^{k+1}$ rate for isotropic meshes. Fig. \ref{fig:C.hexa_skewed} confirms this improvement, albeit in a non-uniform way.

\begin{figure}\centering
  \ref{conv.testC}
  \vspace{0.5cm}\\
  \begin{minipage}[b]{0.45\linewidth}
    \centering
    \begin{tikzpicture}[scale=0.65]
      \begin{loglogaxis}[ legend columns=-1, legend to name=conv.testC ]
        \addplot table[x=meshsize,y=EnergyError] {dat/testC/k0.hexa_straight-skewed.dat};
        \addplot table[x=meshsize,y=EnergyError] {dat/testC/k1.hexa_straight-skewed.dat};
        \addplot table[x=meshsize,y=EnergyError] {dat/testC/k2.hexa_straight-skewed.dat};
        \addplot table[x=meshsize,y=EnergyError] {dat/testC/k3.hexa_straight-skewed.dat};
        \logLogSlopeTriangle{0.90}{0.4}{0.1}{1}{black};
        \logLogSlopeTriangle{0.90}{0.4}{0.1}{2}{black};
        \logLogSlopeTriangle{0.90}{0.4}{0.1}{3}{black};
        \logLogSlopeTriangle{0.90}{0.4}{0.1}{4}{black};
        \legend{$k=0$,$k=1$,$k=2$,$k=3$};
      \end{loglogaxis}
    \end{tikzpicture}
    \subcaption{$E_{\mathrm{a},{\diff[]},h}$ vs. $h$.\label{fig:C.hexa_skewed.energy}}
  \end{minipage}
  \begin{minipage}[b]{0.45\linewidth}
    \centering
    \begin{tikzpicture}[scale=0.65]
      \begin{loglogaxis}[ legend columns=-1, legend to name=conv.testC ]
        \addplot table[x=meshsize,y=H1error] {dat/testC/k0.hexa_straight-skewed.dat};
        \addplot table[x=meshsize,y=H1error] {dat/testC/k1.hexa_straight-skewed.dat};
        \addplot table[x=meshsize,y=H1error] {dat/testC/k2.hexa_straight-skewed.dat};
        \addplot table[x=meshsize,y=H1error] {dat/testC/k3.hexa_straight-skewed.dat};
        \logLogSlopeTriangle{0.90}{0.4}{0.1}{1}{black};
        \logLogSlopeTriangle{0.90}{0.4}{0.1}{2}{black};
        \logLogSlopeTriangle{0.90}{0.4}{0.1}{3}{black};
        \logLogSlopeTriangle{0.90}{0.4}{0.1}{4}{black};
        \legend{$k=0$,$k=1$,$k=2$,$k=3$};
      \end{loglogaxis}
    \end{tikzpicture}
    \subcaption{$E_{1,h}$ vs. $h$.\label{fig:C.hexa_skewed.H1}}
  \end{minipage}
\caption{Errors in Test C for the family of skewed hexagonal meshes. The slopes correspond to the $h^{k+1}$ rates expected for non-skewed meshes.
\label{fig:C.hexa_skewed}}
\end{figure}
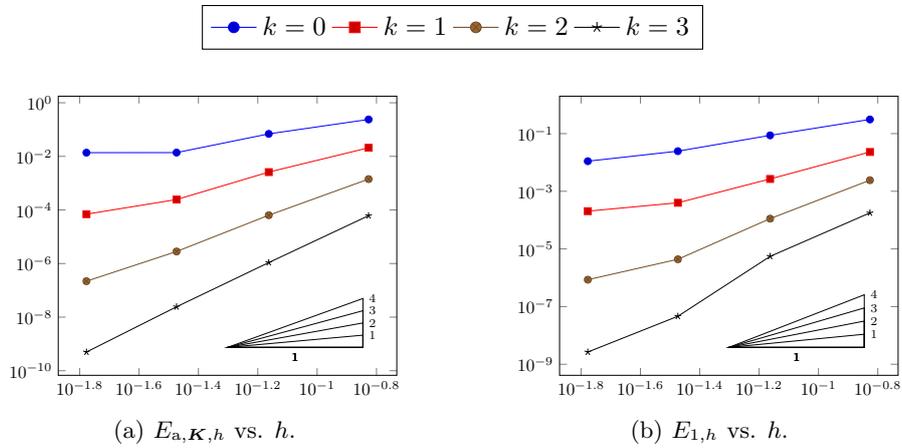

The improvement is clearer if we superimpose the errors for the families of regular and skewed meshes, see Figs. \ref{fig:C.hexa_comp.h.k0k1} and \ref{fig:C.hexa_comp.h.k2k3}. For a given meshsize, selecting a mesh that is stretched in the direction of strong diffusion improves the convergence in both norms; this gain is valid for all degrees, but especially prominent for the lowest-order case $k=0$ (for which, at the considered meshsizes, there is no apparent convergence on non-stretched meshes). In Figs. \ref{fig:C.hexa_comp.dofs.k0k1} and \ref{fig:C.hexa_comp.dofs.k2k3} the same errors are plotted against the number of globally coupled degrees of freedom, which for HHO schemes correspond to the edge unknowns (the element unknowns can be eliminated by static condensation \cite[Appendix B]{hho-book}). In terms of errors vs. number of degrees of freedom, the gain in using skewed meshes is less clear, except for $k=0$; the reason is that meshes entirely made of stretched elements usually have, for a given meshsize, more edges than regular meshes.

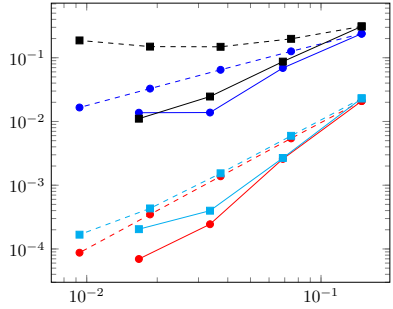
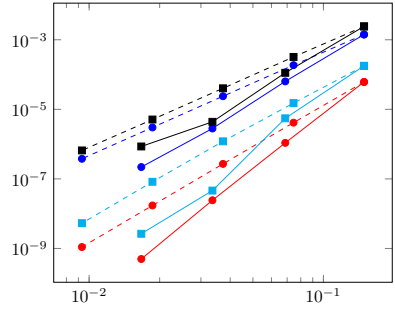
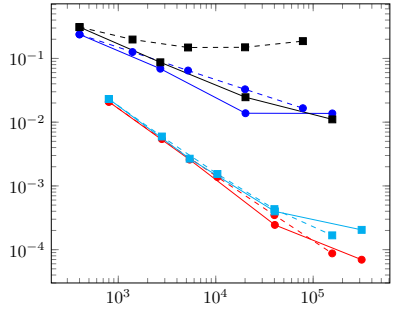
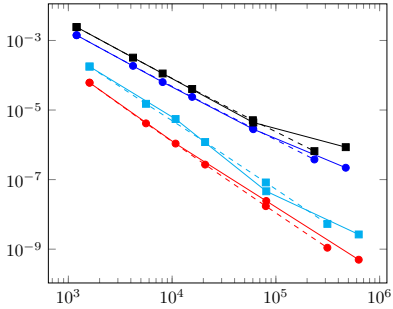
\begin{figure}\centering
  \begin{minipage}[b]{0.45\linewidth}
    \centering
    \begin{tikzpicture}[scale=0.65]
      \begin{loglogaxis}[ legend columns=-1, legend to name=conv.testC.2 ]
        \addplot [mark=*, blue] table[x=meshsize,y=EnergyError] {dat/testC/k0.hexa_straight-skewed.dat};
        \addplot [mark=*, mark options=solid, blue, dashed] table[x=meshsize,y=EnergyError] {dat/testC/k0.hexa_straight-regular.dat};
        \addplot [mark=square*, black] table[x=meshsize,y=H1error] {dat/testC/k0.hexa_straight-skewed.dat};
        \addplot [mark=square*, mark options=solid, black, dashed] table[x=meshsize,y=H1error] {dat/testC/k0.hexa_straight-regular.dat};
        \addplot [mark=*, red] table[x=meshsize,y=EnergyError] {dat/testC/k1.hexa_straight-skewed.dat};
        \addplot [mark=*, mark options=solid, red, dashed] table[x=meshsize,y=EnergyError] {dat/testC/k1.hexa_straight-regular.dat};
        \addplot [mark=square*, cyan] table[x=meshsize,y=H1error] {dat/testC/k1.hexa_straight-skewed.dat};
        \addplot [mark=square*, mark options=solid, cyan, dashed] table[x=meshsize,y=H1error] {dat/testC/k1.hexa_straight-regular.dat};
      \end{loglogaxis}
    \end{tikzpicture}
    \subcaption{Errors vs. $h$ for $k=0$ (top four plots) and $k=1$ (bottom four plots).\label{fig:C.hexa_comp.h.k0k1}}
  \end{minipage}
  \hspace*{1em}
  \begin{minipage}[b]{0.45\linewidth}
    \centering
    \begin{tikzpicture}[scale=0.65]
      \begin{loglogaxis}[ legend columns=-1, legend to name=conv.testC.2 ]
        \addplot [mark=*, blue] table[x=meshsize,y=EnergyError] {dat/testC/k2.hexa_straight-skewed.dat};
        \addplot [mark=*, mark options=solid, blue, dashed] table[x=meshsize,y=EnergyError] {dat/testC/k2.hexa_straight-regular.dat};
        \addplot [mark=square*, black] table[x=meshsize,y=H1error] {dat/testC/k2.hexa_straight-skewed.dat};
        \addplot [mark=square*, mark options=solid, black, dashed] table[x=meshsize,y=H1error] {dat/testC/k2.hexa_straight-regular.dat};
        \addplot [mark=*, red] table[x=meshsize,y=EnergyError] {dat/testC/k3.hexa_straight-skewed.dat};
        \addplot [mark=*, mark options=solid, red, dashed] table[x=meshsize,y=EnergyError] {dat/testC/k3.hexa_straight-regular.dat};
        \addplot [mark=square*, cyan] table[x=meshsize,y=H1error] {dat/testC/k3.hexa_straight-skewed.dat};
        \addplot [mark=square*, mark options=solid, cyan, dashed] table[x=meshsize,y=H1error] {dat/testC/k3.hexa_straight-regular.dat};
      \end{loglogaxis}
    \end{tikzpicture}
    \subcaption{Errors vs. $h$ for $k=2$ (top four plots) and $k=3$ (bottom four plots).\label{fig:C.hexa_comp.h.k2k3}}
  \end{minipage}\\
  \vspace*{1em}
  \begin{minipage}[b]{0.45\linewidth}
    \centering
    \begin{tikzpicture}[scale=0.65]
      \begin{loglogaxis}[ legend columns=-1, legend to name=conv.testC.2 ]
        \addplot [mark=*, blue] table[x=NbEdgeDOFs,y=EnergyError] {dat/testC/k0.hexa_straight-skewed.dat};
        \addplot [mark=*, mark options=solid, blue, dashed] table[x=NbEdgeDOFs,y=EnergyError] {dat/testC/k0.hexa_straight-regular.dat};
        \addplot [mark=square*, black] table[x=NbEdgeDOFs,y=H1error] {dat/testC/k0.hexa_straight-skewed.dat};
        \addplot [mark=square*, mark options=solid, black, dashed] table[x=NbEdgeDOFs,y=H1error] {dat/testC/k0.hexa_straight-regular.dat};
        \addplot [mark=*, red] table[x=NbEdgeDOFs,y=EnergyError] {dat/testC/k1.hexa_straight-skewed.dat};
        \addplot [mark=*, mark options=solid, red, dashed] table[x=NbEdgeDOFs,y=EnergyError] {dat/testC/k1.hexa_straight-regular.dat};
        \addplot [mark=square*, cyan] table[x=NbEdgeDOFs,y=H1error] {dat/testC/k1.hexa_straight-skewed.dat};
        \addplot [mark=square*, mark options=solid, cyan, dashed] table[x=NbEdgeDOFs,y=H1error] {dat/testC/k1.hexa_straight-regular.dat};
      \end{loglogaxis}
    \end{tikzpicture}
    \subcaption{Errors vs. nb DOFs for $k=0$ (top four plots) and $k=1$ (bottom four plots).\label{fig:C.hexa_comp.dofs.k0k1}}
  \end{minipage}
  \hspace*{1em}
  \begin{minipage}[b]{0.45\linewidth}
    \centering
    \begin{tikzpicture}[scale=0.65]
      \begin{loglogaxis}[ legend columns=-1, legend to name=conv.testC.2 ]
        \addplot [mark=*, blue] table[x=NbEdgeDOFs,y=EnergyError] {dat/testC/k2.hexa_straight-skewed.dat};
        \addplot [mark=*, mark options=solid, blue, dashed] table[x=NbEdgeDOFs,y=EnergyError] {dat/testC/k2.hexa_straight-regular.dat};
        \addplot [mark=square*, black] table[x=NbEdgeDOFs,y=H1error] {dat/testC/k2.hexa_straight-skewed.dat};
        \addplot [mark=square*, mark options=solid, black, dashed] table[x=NbEdgeDOFs,y=H1error] {dat/testC/k2.hexa_straight-regular.dat};
        \addplot [mark=*, red] table[x=NbEdgeDOFs,y=EnergyError] {dat/testC/k3.hexa_straight-skewed.dat};
        \addplot [mark=*, mark options=solid, red, dashed] table[x=NbEdgeDOFs,y=EnergyError] {dat/testC/k3.hexa_straight-regular.dat};
        \addplot [mark=square*, cyan] table[x=NbEdgeDOFs,y=H1error] {dat/testC/k3.hexa_straight-skewed.dat};
        \addplot [mark=square*, mark options=solid, cyan, dashed] table[x=NbEdgeDOFs,y=H1error] {dat/testC/k3.hexa_straight-regular.dat};
      \end{loglogaxis}
    \end{tikzpicture}
    \subcaption{Errors vs. nb DOFs for $k=2$ (top four plots) and $k=3$ (bottom four plots).\label{fig:C.hexa_comp.dofs.k2k3}}
  \end{minipage}
\caption{Test C: comparison between regular (dashed lines) and skewed (continuous lines) hexagonal meshes. Round markers: $E_{\mathrm{a},\diff[],h}$; square markers: $E_{1,h}$.
\label{fig:C.hexa_comp}}
\end{figure}

\section{Conclusion}\label{sec:conclusion}

We presented a theoretical and numerical study of the accuracy and robustness of the classical HHO method, when applied to anisotropic diffusion equations on distorted meshes. We defined a notion of mesh sequences that accepts in particular elements that become more and more elongated as the mesh is refined, and we established an error estimate that tracks the dependency of the constants with respect to the local diffusion anisotropy and elements skewness. We then presented the results of several numerical tests designed to explore the optimality of the error estimate. These results indicate that some behaviours highlighted by the theoretical estimate (such as the interplay between diffusion anisotropy and mesh skewness) are perceptible in practical numerical results, but they also show that this estimate appears to be pessimistic in its prediction of the behaviour of the error in case of strong anisotropy or skewness.

Further work remains to be done to obtain more optimal estimates in terms of dependency with respect to the tensor anisotropy (this only has to be done for non-skewed meshes, as our approach would then provide an optimal estimate for skewed meshes). An aspect that is not covered by our definition of regular skewed 
mesh sequences is the case of small edges/faces in otherwise isotropic elements; another approach has to be adopted to derive error estimates in such situations. Finally, even though the standard HHO scheme displays some level of robustness on distorted meshes, it would be interesting to develop a variant that is specifically adapted to such meshes, and leads to better condition numbers than the standard method.

\medskip

\textbf{Acknowledgements}
This work was partially supported by the Australian Government through the Australian Research Council's Discovery Projects funding scheme (project number DP170100605).


\bibliographystyle{plain}
\bibliography{hho-aniso}

\begin{thebibliography}{10}

\bibitem{HArDCore}
{HArDCore} -- {H}ybrid {A}rbitrary {D}egree::{C}ore. {V}ersion 2.0.
  https://github.com/jdroniou/hardcore.

\bibitem{ABVW17}
Paola~F. Antonietti, Stefano Berrone, Marco Verani, and Steffen Wei\ss~er.
\newblock The virtual element method on anisotropic polygonal discretizations.
\newblock In {\em Numerical mathematics and advanced applications---{ENUMATH}
  2017}, volume 126 of {\em Lect. Notes Comput. Sci. Eng.}, pages 725--733.
  Springer, Cham, 2019.

\bibitem{Beirao-da-Veiga.Brezzi.ea:13}
L.~Beir\~{a}o~da Veiga, F.~Brezzi, A.~Cangiani, G.~Manzini, L.~D. Marini, and
  A.~Russo.
\newblock Basic principles of virtual element methods.
\newblock {\em Math. Models Methods Appl. Sci. (M3AS)}, 199(23):199--214, 2013.

\bibitem{Cockburn.Gopalakrishnan.ea:09}
B.~Cockburn, J.~Gopalakrishnan, and R.~Lazarov.
\newblock Unified hybridization of discontinuous {G}alerkin, mixed, and
  continuous {G}alerkin methods for second order elliptic problems.
\newblock {\em SIAM J. Numer. Anal.}, 47(2):1319--1365, 2009.

\bibitem{Di-Pietro.Droniou:18}
D.~A. Di~Pietro and J.~Droniou.
\newblock A third {Strang} lemma for schemes in fully discrete formulation.
\newblock {\em Calcolo}, 55(40), 2018.

\bibitem{Di-Pietro.Ern:12}
D.~A. Di~Pietro and A.~Ern.
\newblock {\em Mathematical aspects of discontinuous {G}alerkin methods},
  volume~69 of {\em Math\'ematiques \& Applications (Berlin)}.
\newblock Springer, Heidelberg, 2012.

\bibitem{Di-Pietro.Ern.ea:14}
D.~A. Di~Pietro, A.~Ern, and S.~Lemaire.
\newblock An arbitrary-order and compact-stencil discretization of diffusion on
  general meshes based on local reconstruction operators.
\newblock {\em Comput. Meth. Appl. Math.}, 14(4):461--472, 2014.

\bibitem{hho-book}
Daniele~Antonio Di~Pietro and J\'er\^ome Droniou.
\newblock {\em The Hybrid High-Order Method for Polytopal Meshes: Design,
  Analysis, and Applications}.
\newblock Modeling, Simulation and Applications. Springer International
  Publishing, 2020.
\newblock To appear.

\bibitem{Droniou.Eymard.ea:10}
J.~Droniou, R.~Eymard, T.~Gallou\"{e}t, and R.~Herbin.
\newblock A unified approach to mimetic finite difference, hybrid finite volume
  and mixed finite volume methods.
\newblock {\em Math. Models Methods Appl. Sci. (M3AS)}, 20(2):1--31, 2010.

\bibitem{Herbin.Hubert:08}
R.~Herbin and F.~Hubert.
\newblock Benchmark on discretization schemes for anisotropic diffusion
  problems on general grids.
\newblock In R.~Eymard and J.-M. H\'{e}rard, editors, {\em Finite Volumes for
  Complex Applications V}, pages 659--692. John Wiley \& Sons, 2008.

\bibitem{Wang.Ye:13}
Junping Wang and Xiu Ye.
\newblock A weak {G}alerkin finite element method for second-order elliptic
  problems.
\newblock {\em J. Comput. Appl. Math.}, 241:103--115, 2013.

\bibitem{W19}
Steffen Wei\ss{}er.
\newblock Anisotropic polygonal and polyhedral discretizations in finite
  element analysis.
\newblock {\em ESAIM Math. Model. Numer. Anal.}, 53(2):475--501, 2019.

\end{thebibliography}
\end{document}